\documentclass[12pt,reqno]{amsart}

\usepackage{graphicx,amsfonts,amssymb,amsmath,amsthm,xpatch}
\usepackage{fullpage}
\usepackage[svgnames]{xcolor}
\usepackage{hyperref}
\hypersetup{
  colorlinks = true,
  linkcolor =red,
  anchorcolor = red,
  citecolor = DarkGreen,
  urlcolor = blue
}
\usepackage[all,arc]{xy}
\usepackage{tikz-cd,float}
\usepackage{mathrsfs,mathtools}
\usepackage{caption}
\usepackage{enumitem}
\usepackage{pdflscape,bm}
\usepackage{subfig}

\newtheorem{theo}{{\sc Theorem}}[section]

\newtheorem{defn}[theo]{{\sc Definition}}

\newtheorem{lem}[theo]{{\sc Lemma}}

\theoremstyle{definition}

\newtheorem{rem}[theo]{{\it Remark}}

\allowdisplaybreaks

\makeatletter
\def\blfootnote{\gdef\@thefnmark{}\@footnotetext}
\def\@cite#1#2{[\textbf{#1}\if@tempswa, #2\fi]}
\newcounter{proofpart}
\xpretocmd{\proof}{\setcounter{proofpart}{0}}{}{}
\newcommand{\proofpart}[2]{%
  \par
  \addvspace{\medskipamount}%
  \stepcounter{proofpart}%
  \noindent\emph{#1 \theproofpart: #2}\par\nobreak\smallskip
  \@afterheading
}
\makeatother

\DeclareFontFamily{U}{mathx}{\hyphenchar\font45}
\DeclareFontShape{U}{mathx}{m}{n}{
      <5> <6> <7> <8> <9> <10>
      <10.95> <12> <14.4> <17.28> <20.74> <24.88>
      mathx10
      }{}
\DeclareSymbolFont{mathx}{U}{mathx}{m}{n}
\DeclareFontSubstitution{U}{mathx}{m}{n}
\DeclareMathAccent{\widecheck}{0}{mathx}{"71}
\DeclareMathAccent{\wideparen}{0}{mathx}{"75}

\title{Semiclassical Measures of Eigenfunctions of the attractive Coulomb operator}

\author{Nicholas Lohr}
\address{Northwestern University, Evanston, IL, 60208}
\email{nlohr@math.northwestern.edu}
\date{\today}

\begin{document}
\blfootnote{The author was partially supported by NSF RTG grant
  DMS-2136217.\par
  The author has no competing interests to declare that are relevant to the content of this article.}
\begin{abstract}
We characterize the set of semiclassical measures corresponding to sequences of eigenfunctions of the attractive Coulomb operator $\widehat{H}_{\hbar}\coloneqq-\frac{\hbar^2}{2}\Delta_{\mathbb{R}^3}-\frac{1}{|x|}$. In particular, any Radon probability measure on the fixed negative energy hypersurface $\Sigma_E$ of the Kepler Hamiltonian $H$ in classical phase space that is invariant under the regularized Kepler flow is the semiclassical measure of a sequence of eigenfunctions of $\widehat{H}_{\hbar}$ with eigenvalue $E$ as $\hbar \to 0$. The main tool that we use is the celebrated Fock unitary conjugation map between eigenspaces of $\widehat{H}_{\hbar}$ and $-\Delta_{\mathbb{S}^3}$. We first prove that for any Kepler orbit $\gamma$ on $\Sigma_E$, there is a sequence of eigenfunctions that converge in the sense of semiclassical measures to the delta measure supported on $\gamma$ as $\hbar \to 0$, and we finish using a density argument in the weak-* topology. 
\end{abstract}
\maketitle
\section{Introduction}
In this article, we characterize the semiclassical
measures corresponding to eigenfunctions of the attractive Coulomb operator, defined as
\begin{equation}\label{eq:operator}
\widehat{H}_{\hbar}:L^2(\mathbb{R}^3) \to L^2(\mathbb{R}^3),\quad \widehat{H}_{\hbar}\coloneqq -\frac{\hbar^2}{2}\Delta-\frac{1}{|x|},\quad  \hbar>0.
\end{equation}
This operator is the first approximation of the quantum
hydrogen atom. That is, fixing the reduced mass of the
electron-proton system to $1$, the reduced Bohr radius to
$\hbar^2$, and ignoring all relativistic and spin-coupling
effects, the Schr\"odinger operator for the relative position of
the electron is given by $\widehat{H}_{\hbar}$. It is well-known
that $\widehat{H}_{\hbar}$ is self-adjoint on $L^2(\mathbb{R}^3)$
with domain $H^2(\mathbb{R}^3)$, and the spectrum of
$\widehat{H}_{\hbar}$ is bounded from below (see \cite[Theorem
9.38]{H13} for a proof using the Kato-Rellich theorem and
\cite[Chapter 8, \S 7]{T11} for a different proof using the
Friedrichs method with Hardy's inequality). In fact, the
spectrum of $\widehat{H}_{\hbar}$ decomposes into a negative pure point spectrum and a nonnegative continuous spectrum completely explicitly:
\begin{equation}\label{eq:energy}
\operatorname{spec}\widehat{H}_{\hbar} =\Big\{E_{N}(\hbar) \coloneqq  -\frac{1}{2\hbar^2(N+1)^2} \mid N=0,1,\ldots\Big\}\sqcup [0,\infty).
\end{equation}
Each eigenvalue $E_N(\hbar)$ has multiplicity $(N+1)^2$ (see
\cite[Theorem 18.4, Corollary 18.5]{H13}). 
\par The attractive Coulomb operator corresponds to the
classical phase space Hamiltonian
\begin{equation*}
H:T^*(\mathbb{R}^3\setminus\{0\}) \to \mathbb{R},\quad H(x,\xi)\coloneqq \frac{|\xi|^2}{2}-\frac{1}{|x|},
\end{equation*}
called the Kepler Hamiltonian, where we identify
$T^*(\mathbb{R}^3\setminus\{0\})=\mathbb{R}^3\setminus\{0\}
\times \mathbb{R}^3$ using the Riemannian metric on
$\mathbb{R}^3\setminus\{0\}$. For a fixed energy $E$, the
Hamiltonian orbits, also called Kepler orbits, lie on the energy
hypersurface
\begin{equation*}
\Sigma_E \coloneqq \{(x,\xi) \in T^*(\mathbb{R}^3 \setminus \{0\}) \mid H(x,\xi)=E\}.
\end{equation*}
For any energy $E \in \mathbb{R}$, $\Sigma_E$ is \textit{not} compact due to the $x \to 0, |\xi| \to \infty$ regime.
For $E<0$, the orbits consist of two types: periodic orbits
whose configuration space projections are planar ellipses, and
unbounded ``collision'' orbits whose configuration space projections
are line segments terminating at the origin in finite
time. The configuration space projections of the periodic
Kepler orbits follow Kepler's laws of planetary motion (with one
body fixed and all physical constants fixed to $1$). Namely, the periodic configuration space trajectories
\begin{itemize}
\item are ellipses with the origin fixed at one focus,
\item are such that the line segment connecting the trajectory
to the origin sweeps out equal areas during equal time
intervals,
\item have period $T$ related to the energy $E$ by the
formula
\begin{equation}\label{eq:kepler3}
T=\frac{2\pi}{p_0^3},\quad p_0 \coloneqq \sqrt{-2E},
\end{equation}
\end{itemize}
where we have used our convention on physical constants. Observe
that Kepler’s third law is popularly stated with the length of
the semi-major axis $a$, but, with our conventions, $a=p_0^{-2}$
(see \cite[(5)]{M83} and the very nice expository article
\cite{vHH09}).
\par This Hamiltonian system is not only completely integrable, but it is maximally superintegrable with $5$ independent integrals of motion coming from the components of the
conserved quantities of the Hamiltonian $H$, the angular momentum vector $L$, and
the Runge-Lenz eccentricity vector $R$ defined by
\begin{equation}\label{eq:rungelenz}
L(x,\xi) \coloneqq x \times \xi,\qquad R(x,\xi)=\Big(|\xi|^2-\frac{1}{|x|} \Big)x-(x\cdot \xi)\xi.
\end{equation}
On $\Sigma_E$, the magnitudes of these quantities are related by
the formula
$$|R|^2=1+2E|L|^2. $$
A Kepler orbit is a collision orbit if and only if $L=0$. Provided that $L\neq 0$, in configuration space, $L$ determines the plane of motion, $|R|$ is the eccentricity of the ellipse, $R$ and the foci are colinear, and $|2E|^{-1}$
is the length
of the semi-major axis (as noted previously). The Runge-Lenz vector $R$ has a long, complicated history of discovery and rediscovery (see the works of Goldstein \cite{G75,G76}), but,
most noteworthy, Hamilton in \cite{H47} showed that the Runge-Lenz vector can be understood as coming from the geometry of the momentum space projections of the Kepler orbits,
which miraculously happen to be circles. Each circle has radius $1/|L|$ and is centered at the point obtained by rotating $R/|L|$ by $90^{\circ}$ in the plane of motion (more
carefully, these circles degenerate into lines for the collision orbits). The superintegrability
explains why the bounded orbits are not merely quasi-periodic and confined to invariant tori
as guaranteed from the Liouville-Arnold theorem (see \cite[Chapter 10]{A89}), but the bounded
orbits are genuinely \textit{periodic} (see \cite{GS90} for more on the symmetries of this problem).
\par However, because of the collision orbits, the Hamiltonian
flow of $H$ is \textit{not} complete. In \cite{M70}, Moser
compactified $\Sigma_E$ to a manifold $\overline{\Sigma_E}$
(defined in \eqref{eq:compact}) where the Hamiltonian flow is
regularized by a reflection condition. Roughly speaking, when the collision orbits hit the origin, they are reflected back along the same line, resembling a degenerate ellipse. The manifold $\overline{\Sigma_E}$ is diffeomorphic to $T_1^*( \mathbb{S}^3)$, and, up to a reparametrization, the regularized Hamiltonian flow maps to the cogeodesic flow on $\mathbb{S}^3$. In particular, the collision orbits are mapped to the great circles passing through the `north pole' of $\mathbb{S}^3$. This completes the Hamiltonian flow and extends the collision orbits past their finite collision time to be periodic on all time and obeying Kepler's third law \eqref{eq:kepler3}.
\par For fixed $E<0$ and sequences $\hbar_j \to 0$, $N_j \to \infty$
satisfying $E_{N_j}(\hbar_j)\to E$, we say that a sequence $\Psi_j$
of $L^2$-normalized eigenfunctions of $\widehat{H}_{\hbar_j}$
satisfying
\begin{equation*}
\widehat{H}_{\hbar_j}\Psi_j=E_{N_j}(\hbar_j) \Psi_j
\end{equation*}
converges to a nonnegative Radon measure $\mu$ on $T^*\mathbb{R}^3$ in the sense of semiclassical
measures if, for any $a \in C_c^{\infty}(T^*\mathbb{R}^3)$, we
have
\begin{equation*}
\langle \operatorname{Op}_{\hbar_j}(a)\Psi_{j},\Psi_{j}
\rangle \xrightarrow{j \to \infty }\int_{T^*\mathbb{R}^3}a(x,\xi)
d\mu(x,\xi), 
\end{equation*}
where $\operatorname{Op}_{\hbar}$ denotes semiclassical Weyl
quantization (see \cite[\S E.3.]{DZ19} for more on semiclassical
measures). For any regularized Kepler orbit $\overline{\gamma}$ on $\overline{\Sigma_E}$, there exists sequences $\hbar_j \to 0$, $N_j \to \infty$ satisfying $E_{N_j}(\hbar_j)\to E$ and $L^2$-normalized eigenfunctions $\Psi_{\hbar_j,N_j}^{\gamma}$ of $\widehat{H}_{\hbar_j}$, called Coulomb coherent states,
which satisfy
\begin{equation*}
\widehat{H}_{\hbar}\Psi_{\hbar_j,N_j}^{\gamma}=E_{N_j}(\hbar_j)\Psi_{\hbar_j,N_j}^{\gamma},
\end{equation*}
and converge to the delta measure supported on $\gamma$
in the sense of semiclassical measures (see Definition \ref{def:fergie}). That is, we have the following theorem:
\subsection{Statement of Results}
\begin{theo}\label{theo:1}
Let $E<0$ and $a \in C_c^{\infty}(T^*\mathbb{R}^3)$. If $\overline{\gamma}$
  is a regularized Kepler orbit on the energy hypersurface $\overline{\Sigma_E}$, then there exists sequences $\hbar_j \to 0$, $N_j \to \infty$ satisfying $E_{N_j}(\hbar_j)\to E$ such that 
$$
\langle
\operatorname{Op}_{\hbar_j}(a)\Psi_{\hbar_j,N_j}^{\gamma},\Psi_{\hbar_j,N_j}^{\gamma} \rangle
\xrightarrow{j \to \infty}\frac{p_0^3}{2\pi}
\int_{0}^{2\pi/p_0^3}\overline{a}(\overline{\gamma}(t))dt,
$$
where $p_0 \coloneqq \sqrt{-2E}$, $2\pi/p_0^3$ is the period of the regularized Kepler orbits, $\operatorname{Op}_{\hbar}$
denotes semi-classical Weyl quantization, and $\overline{a}$ is defined in \eqref{eq:barr}.
\end{theo}
Using Theorem \ref{theo:1}, we prove the main result of the
article, Theorem \ref{theo:2}. As noted previously, the Hamiltonian flow on $\Sigma_E$ is \textit{not} complete, so we have to carefully define what it means for a measure on $\Sigma_E$ to be invariant under the Hamiltonian flow. We say that a Radon probability
measure $\mu$ on $\Sigma_E$ is invariant under the Hamiltonian
flow if the pushforward measure $(i_{\Sigma_E})_*\mu$ is invariant under the regularized Hamiltonian flow $\overline{\Xi_H^{\bullet}}$ (defined in \eqref{eq:reghamil}) where $i_{\Sigma_E}:\Sigma_E\to \overline{\Sigma_E}$ is the inclusion map (defined in \eqref{eq:inclusion}).
\begin{theo}\label{theo:2}
Let $E<0$ and let $\mu$ be a Radon probability
measure on $\Sigma_E$ invariant under the Hamiltonian flow. Then
$\mu$ is a semiclassical measure of a sequence $\Psi_{j}$
of eigenfunctions of $\widehat{H}_{\hbar}$. That is, there exists sequences $\hbar_j \to 0$, $N_j \to \infty$ satisfying $E_{N_j}(\hbar_j)\to E$ and $L^2$-normalized $\Psi_j$ such that $\widehat{H}_{\hbar_j}\Psi_{j}=E_{N_j}(\hbar_j)\Psi_{j}$ and
$$\langle \operatorname{Op}_{\hbar_j}(a)\Psi_{j},\Psi_{j} \rangle
\xrightarrow{j \to \infty}\int_{\Sigma_E}a(x,\xi) d\mu(x,\xi),  $$
for any $a \in C_c^{\infty}(T^*\mathbb{R}^3)$.
\end{theo}
\begin{rem}\label{rem:notcompact}
We, again, emphasize that $\Sigma_E$ is \textit{not} compact. The theorem applies to measures $\mu$ not necessarily compactly supported in $\Sigma_E$ (e.g. delta measures supported on collision orbits) with the caveat that the support of $a$ is compact in phase space. All of the collision orbits will leave the support of $a$ as they approach the collision point, and it is of interest to investigate what happens if $a$ has support near $x \to 0,|\xi| \to \infty$. In this case, one has to be careful with the very definition of convergence in the sense of semiclassical measures.
\end{rem}
\begin{rem}\label{rem:higher}
We remark that our methods apply to the analogous statements of Theorems \ref{theo:1} and \ref{theo:2} for the operator $-\frac{\hbar^2}{2}\Delta_{\mathbb{R}^d}-\frac{1}{|x|}$ on $L^2(\mathbb{R}^d)$ for $d \geq 3$. Indeed, the Moser and Fock maps (and, consequently, the Coulomb coherent states) naturally extend to any dimension greater than or equal to three (see \cite{HdL12} for the Moser map, for example). Due to physical relevance and notational convenience, we stick to $d=3$.
  \end{rem}
The strategy of the proofs of these theorems is very similar to
the methods used in \cite{JZ99}. Indeed, for Theorem
\ref{theo:1}, we construct the Coulomb coherent states by
applying Fock's unitary map (the `quantization' of Moser's regularization) to the
highest weight spherical harmonics on the 3-sphere,
$\mathbb{S}^3$, which concentrate on great circles. We then
approximate all invariant measures by convex combinations of
delta measures supported on Kepler orbits.
\par In general, it is hard to characterize the set of all
semiclassical measures $\mu$ for a given operator.  In the non-chaotic setting, the set of semiclassical measures has been completely characterized in a few settings, including the Laplace-Beltrami operator on compact rank-one symmetric spaces \cite{M08} (which includes spheres \cite{JZ99}), space forms \cite{AM10}, the 2-torus \cite{J97}, and, more recently, harmonic oscillators on $\mathbb{R}^d$ \cite{A20,S19,AM22}. These are all examples of completely integrable systems, and asymptotically vanishing perturbations of these and other systems have also been well-studied. The literature is vast in this generality; see the introductions of \cite{A20,AM22} for accounts of the literature. In the case of small, non-asymptotically vanishing perturbations of completely integrable systems
(KAM systems), we have the few recent works of \cite{A20,GH22,G23}.
\par In the chaotic setting, the set of semiclassical measures is \textit{almost} characterized by the quantum ergodicity theorem. One instance of this theorem is the following: if $M$ is
a compact, smooth Riemannian manifold without boundary such that
cogeodesic flow is ergodic with respect to the Liouville measure, then any orthonormal sequences of eigenfunctions of the semiclassical Laplace-Beltrami operator of $M$ with eigenvalue of 1 admit a density-1 subsequence that converges to the Liouville measure in the sense of semiclassical measures as $\hbar \to 0$ (see \cite{S74a,S74b,L93,Z87,CdV85} for the original works and \cite{D22} for an exposition of the results in the chaotic setting). The quantum unique ergodicity conjecture states that it is not necessary to descend to a density-1 subsequence and thus completely characterizes the set of semiclassical measures in this setting, but this conjecture is still open.
\par Sequences of eigenfunctions of $\widehat{H}_{\hbar}$ that
concentrate on classical trajectories were first studied in
\cite{GDB89,N89} and later in \cite{K96,TV-B97}. Theorem
\ref{theo:1} was proved in configuration and momentum space
separately in \cite{TV-B97} for periodic, non-collision
orbits. 
\par The implications of the regularized Hamiltonian flow on the quantum dynamics of Schr\"{o}dinger operators with Coulomb-like potentials has also been well-studied. G\'{e}rard and Knauf in \cite{GK91} showed that the semiclassical wavefront set of time-dependent Schr\"{o}dinger equation solutions $u_{\hbar}(t)=e^{-it \widehat{H}_{\hbar}/\hbar}u_{\hbar,0}, u_{\hbar,0} \in L^2(\mathbb{R}^3)$ propagates along regularized Hamiltonian orbits, including beyond the collision time. Additionally, Keraani in \cite{K05} showed the analogous statement for the propagation of semiclassical measures initially supported away from the origin. These papers regularize the Hamiltonian flow through the Kustaanheimo-Stiefel (KS) transformation. The KS map reduces this three-dimensional Hamiltonian flow to a suitably constrained four-dimensional harmonic oscillator flow (see the original works of \cite{K64,KS65} as well as the book \cite{SS71}), and it is the three-dimensional generalization of the one-dimensional and two-dimensional regularizations of the Kepler problem known to Euler \cite{E74} and Levi-Civita \cite{LC20}, respectively. Although the KS transformation has proven to be a powerful tool as exhibited in the aforementioned \cite{GK91,K05} and other work such as \cite{CJK08}, it has several drawbacks. The inverse KS map is only locally defined via introducing a dummy variable defined on the circle, the KS map also has no obvious generalization to dimensions higher than three, and, to the author's knowledge, it has no obvious `quantization' that relates the spectrum of the four-dimensional harmonic oscillator to that of the Coulomb operator. We note that the unitary Fock map has a satisfactory answer to these three defects, and we use these additional properties in this article.
\par The point of this article is to first generalize the concentration results in \cite{TV-B97} to phase space in Theorem \ref{theo:1}. By specializing the potential to be exact Coulomb and utilizing the Moser and Fock maps, we analyze the singularity at the origin through states concentrating on the collision orbits, and we use
this to characterize all of the semiclassical measures of
eigenfunctions of $\widehat{H}_{\hbar}$, which complements the existing results of \cite{K05}. 
\subsection{Future Work}
In future work, we plan to study the finer pointwise asymptotics of the Wigner distributions of the Coulomb coherent states in a similar fashion as in \cite{L23}. We also plan on studying asymptotically vanishing perturbations of the Coulomb system, similarly to what has been done for the sphere \cite{M09,MR19}, Zoll manifolds \cite{M08,MR16}, and harmonic oscillators \cite{AM22}.
\subsection{Acknowledgments}
This article is part of the Ph.D. thesis of the author at Northwestern University under the guidance of Steve Zelditch. The author thanks Jared Wunsch for continued conversations and support after the passing of Steve Zelditch. The author also thanks Erik Hupp, Ruoyu P. T. Wang, and Jeff Xia for helpful conversations, as well as the very thorough and thoughtful anonymous referee.
\subsection{Background: Classical and Quantum Mechanical Mappings between Coulomb and spherical dynamics}
In this section, we introduce the relevant classical and quantum mechanical maps that are involved with this problem. For completeness, we reproduce proofs of basic facts about these maps, and further properties and generalizations to $\mathbb{R}^d$ can be found in \cite{M70,HdL12} for the Moser map and \cite{F35,BI66,RC21} for the Fock map.
 \subsubsection{The Classical Mechanical Moser Map}
 In this section, we define the classical Moser map, first defined by Moser in \cite{M70} (see \cite{HdL12} for an overview). This map regularizes the incomplete Kepler flow by mapping the (regularized) Hamiltonian orbits on a compactified $\Sigma_E$ to the geodesics of $T_1^*\mathbb{S}^3$.
 We use the notation
 $$\mathbb{S}_{\neq \mathsf{NP}}^3\coloneqq \mathbb{S}^3\setminus \{\mathsf{NP}\},\quad \mathsf{NP} \coloneqq (0,0,0,1),$$
 to denote the sphere punctured at the `north pole.'
Let $\omega: \mathbb{R}^3 \to \mathbb{S}_{\neq \mathsf{NP}}^3$
be inverse of stereographic projection from the north pole. That
is, the maps $\omega: \mathbb{R}^3 \to \mathbb{S}_{\neq \mathsf{NP}}^3$ and $\omega^{-1}:\mathbb{S}_{\neq \mathsf{NP}}^3\to \mathbb{R}^3$ are given by
\begin{equation}\label{eq:stereo}
\omega(x) \coloneqq \frac{1}{|x|^2+1}\begin{cases}
2x_k & \text{if }k<4\\

|x|^2-1 &\text{if }k=4
\end{cases}, \quad \omega^{-1}(u)_j=\frac{u_j}{1-u_4},\ j=1,2,3.
\end{equation}
It can be easily computed that the pullback $\omega^*:T^*\mathbb{R}^3 \to T^*(\mathbb{S}_{\neq \mathsf{NP}}^3)$ is
\begin{align}
\omega^*(x,\xi)=(\omega(x),\eta) \quad \text{with} \quad \eta_j =\begin{cases}
           \xi_j \frac{|x|^2+1}{2}-(x\cdot \xi) x_j &\text{ if }
                                               j<4\\
           x \cdot \xi &\text{ if }
                                               j=4
                                                                  \end{cases}\label{eq:mos},
\end{align}
where we have identified $T^* \mathbb{R}^3\cong T \mathbb{R}^3=\mathbb{R}_x^3 \times \mathbb{R}_{\xi}^3$ and $T^*(\mathbb{S}_{\neq \mathsf{NP}}^3)\cong T(\mathbb{S}_{\neq \mathsf{NP}}^3)\subset T \mathbb{R}^4=\mathbb{R}_u^4 \times \mathbb{R}_{\eta}^4$ with the musical isomorphisms induced by the respective Riemannian metrics.
          \begin{defn}\label{def:moser} Let $E<0$ and define $p_0 \coloneqq
          \sqrt{-2E}$. Define the Moser map
          \begin{equation*}
\mathcal{M}_E:T^*\mathbb{R}^3 \to T^*(\mathbb{S}_{\neq \mathsf{NP}}^3),\quad \mathcal{M}_E \coloneqq   \omega^* \circ R_{-\pi/2}\circ S 
          \circ \mathcal{D}_{p_0}
\end{equation*}
 where $\mathcal{D}_{p_0}(x,\xi) \coloneqq (p_0 x
          ,p_0^{-1}\xi)$ is the symplectic dilation by $p_0$,
          $R_{-\pi/2}(x,\xi) \coloneqq (\xi,-x)$ is the symplectic rotation by $-\pi/2$, and $S(x,\xi) \coloneqq (p_0x,\xi)$ is a nonsymplectic dilation. Using
          \eqref{eq:mos}, we can write $\mathcal{M}_E$ explicitly
          as 
          \begin{align}
\mathcal{M}_E(x,\xi)=\big(\omega(p_0^{-1}\xi),\eta\big) \quad \text{where} \quad \eta_j =\begin{cases}
           -x_j \frac{|\xi|^2+p_0^2}{2}+(x\cdot \xi) \xi_j &\text{ if }
                                               j<4\\
           -p_0(x \cdot \xi) &\text{ if }
                                               j=4
                                                                                               \end{cases}.\label{eq:mos2}
          \end{align}
          The inverse $\mathcal{M}_E^{-1}:T^*(\mathbb{S}_{\neq \mathsf{NP}}^3)\to T^*\mathbb{R}^3$ is given by
          \begin{align}
\mathcal{M}_E^{-1}(u,\eta)=(x,p_0\omega^{-1}(u) ) \quad \text{where} \quad x_k=\tfrac{1}{p_0^2}\big(\eta_k(u_4-1)-\eta_4u_k\big) \quad \text{for } k=1,2,3.
\end{align}
\end{defn}
\begin{rem}
One can compute
\begin{equation}\label{eq:sympl}
\mathcal{M}_E^* \Big(\sum_{k=1}^4du_k
\wedge d\eta_k \Big)=p_0\sum_{k=1}^3 dx_k \wedge d\xi_k
\end{equation}
where $\sum_{k=1}^4du_k\wedge d\eta_k$ denotes the symplectic form on $T^*\mathbb{R}^4$ restricted to $T^*(\mathbb{S}_{\neq \mathsf{NP}}^3)$. Additionally, the functions $u_j\eta_k-u_k\eta_j$ on $T^*(\mathbb{S}_{\neq \mathsf{NP}}^3)$ pulled back by $\mathcal{M}_E$ can be computed as
\begin{align}
\mathcal{M}_E^*(u_j\eta_k-u_k\eta_j)&=p_0(x_j\xi_k-x_k\xi_j),\quad
                          j,k \neq 4, \label{eq:preang}\\
  \mathcal{M}_E^*(u_j\eta_4-u_4\eta_j)&=\frac{|\xi|^2-p_0^2}{2}x_j-(x \cdot \xi)\xi_j,\quad
                          j \neq 4 \label{eq:prelenz}.
\end{align}
That is, \eqref{eq:preang} states that $\mathcal{M}_E$ pulls back the components of angular
momentum not involving the fourth coordinate in $\mathbb{R}^4$
to \textit{all} the (scaled) components of angular momentum in
$\mathbb{R}^3$. Put differently, for $g \in
\operatorname{SO}(3)$, we have
\begin{equation}\label{eq:later}
\mathcal{M}_E \circ g^*=\begin{pmatrix} g & 0\\ 0 & 1\end{pmatrix}^* \circ \mathcal{M}_E,
\end{equation}
where the asterisk denotes the symplectic lift of the rotation action on the base manifold to the cotangent bundle. 
\par To further understand \eqref{eq:prelenz}, we
first observe that one can check
$\mathcal{M}_E|_{\Sigma_E}=T_1^*(\mathbb{S}_{\neq
  \mathsf{NP}}^3)$. On $\Sigma_E$, the right hand side of \eqref{eq:prelenz} coincides with the components of $R$ (see \eqref{eq:rungelenz}). Finally, it is worth emphasizing that the Moser map \textit{crucially} depends on the energy level $E$.
\end{rem}
\begin{theo}[\cite{M70}, Theorem 1]\label{theo:moser} Fix $E<0$.
Up to a reparametrization of time, the Moser map $\mathcal{M}_E$ transforms the Kepler flow on $\Sigma_E$ onto the cogeodesic flow on $T_1^*(\mathbb{S}_{\neq \mathsf{NP}}^3)$ parametrized by arc length. More specifically, 
if $\gamma(t)=(x(t),\xi(t)) \in T^*\mathbb{R}^3$ is a Kepler orbit on $\Sigma_E$, then $\varphi(s)=(u(s),\eta(s)) \coloneqq \mathcal{M}_E(\gamma(t(s)) \in T_1^*(\mathbb{S}_{\neq \mathsf{NP}}^3)$ is a 
cogeodesic on $T_1^*(\mathbb{S}_{\neq \mathsf{NP}}^3)$ parametrized by arc length $s$ where
$t(s)$ satisfies
\begin{equation}\label{eq:timechange}
\frac{dt}{ds}=\frac{|x(t(s))|}{p_0}=\frac{1-u_4(s)}{p_0^3},\quad t(0)=0.
\end{equation}
\end{theo}
\begin{rem} Note that $t(s)$ is strictly increasing since $t'(s)> 0$. In fact, if we view \eqref{eq:timechange} as a differential equation defined on all $s \in \mathbb{R}$, then $t(s)$ is increasing since $t'(s)=0$ only at the discrete, periodic points $s$ where $u_4(s)=1$. If we integrate both sides of \eqref{eq:timechange} from $s=0$ to $2\pi$, we recover Kepler's third law \eqref{eq:kepler3} since $u_4(s)=a \cos s +b \sin s$ for some constants $a,b$.
\end{rem}
\begin{rem}\label{rem:collision}
As noted in the introduction proceeding \eqref{eq:rungelenz}, a Kepler orbit is a collision orbit if and only if the angular momentum vector $L=0$. In this case, by \eqref{eq:preang}, we see that the corresponding geodesic on $\mathbb{S}_{\neq \mathsf{NP}}^3$ has zero angular momentum in the directions not involving the fourth coordinate. That is, the collision Kepler orbits correspond to the great circle geodesics terminating at $\mathsf{NP}$, the north pole. If we let $\gamma$ be a collision Kepler orbit, we define $t_{\gamma}$ to be the time at which $\gamma$ blows-up. The Kepler orbit $\gamma$ is defined only on the interval $(t_{\gamma}-\frac{2\pi}{p_0^3},t_{\gamma})$, and Moser's regularization continues $\gamma$ to be $\frac{2\pi}{p_0^3}$ periodic on $\mathbb{R}$ by continuing the corresponding great circle geodesic past the north pole termination point. See Definition \ref{def:moser2} and the proceeding remarks for more rigor.
\end{rem}
\begin{proof}
Let $\mathcal{M}_E(x,\xi)=(u,\eta)$. From \eqref{eq:mos2}, one can
compute
\begin{equation}\label{eq:relate}
\frac{1}{2}|\eta|^2=\frac{|x|^2(|\xi|^2+p_0^2)^2}{8}.
\end{equation}
On $T^*(\mathbb{S}_{\neq \mathsf{NP}}^3)$, define $K(u,\eta) \coloneqq \frac{1}{2}|\eta|^2$. Note that the
Hamiltonian flow of $K$ on the level hypersurface
$\{K=\frac{1}{2}\}$ is the cogeodesic flow on $T_1^*(\mathbb{S}_{\neq \mathsf{NP}}^3)$ parametrized by arc
length time $s$. By \eqref{eq:relate}, the
Hamiltonian orbits of $$F(x,\xi) \coloneqq
\frac{|x|^2(|\xi|^2+p_0^2)^2}{8}$$ on the level
hypersurface $\{F=\frac{1}{2}\}$ parametrized in time parameter $t'$ are images under
$\mathcal{M}_E^{-1}$ of the Hamiltonian orbits of $K$ on the level
hypersurface $\{K=\frac{1}{2}\}$ parametrized by arc length $s$ where
$$\frac{dt'}{ds}=\frac{1}{p_0}. $$
Define
$$G(x,\xi)=\sqrt{2
  F(x,\xi)}-1=\frac{|x|(|\xi|^2+p_0^2)}{2}-1.$$
It is easy to see that the Hamiltonian flow of $F$ on the level
hypersurface $\{F=\frac{1}{2}\}$ is
equivalent to the Hamiltonian flow of $G$ on the level
hypersurface $\{G=0\}$. Finally, note that
$$H(x,\xi)=\frac{1}{|x|}G(x,\xi)-\frac{p_0^2}{2}. $$
Again, it is easy to see that the Hamiltonian flow of $G$ on the
level hypersurface $\{G=0\}$ in the time parameter $t'$ is
equivalent to the Hamiltonian flow of $H$ on
$\{H=-\frac{p_0^2}{2}=E\}$ in the time parameter $t$ where
$\frac{dt}{dt'}=|x(t(t'))|$. Altogether, we have
$$\frac{dt}{ds}=\frac{dt}{dt'}\frac{dt'}{ds}=\frac{|x(t(s))|}{p_0}=\frac{1}{p_0}\frac{2}{|\xi(t(s))|^2+p_0^2}=\frac{1}{p_0}\frac{2}{|p_0\omega^{-1}(u(s))|^2+p_0^2}=\frac{1-u_4(s)}{p_0^3}, $$
and we are done.
\end{proof}
Moser's regularization adds the point $\mathsf{NP}$ to
$T^*(\mathbb{S}_{\neq \mathsf{NP}}^3)$ and thus compactifies
$\Sigma_E$. In order to do this rigorously, we `patch' the
behavior at the south pole to the north pole. Defining
$\mathsf{SP} \coloneqq -\mathsf{NP}=(0,0,0,-1)$, observe the
diagram
\begin{equation}\label{eq:newdiagram}
\begin{tikzcd}
	{(T^*\mathbb{R}^3)\setminus 0} & {T^*(\mathbb{S}_{\neq \mathsf{SP},\mathsf{NP}}^3)} \\
	{(T^*\mathbb{R}^3)\setminus 0} & {T^*(\mathbb{S}_{\neq \mathsf{SP},\mathsf{NP}}^3)}
	\arrow["{\mathcal{M}_E}", from=1-1, to=1-2]
	\arrow["{\mathcal{I}_E}"', from=1-1, to=2-1]
	\arrow["{\mathcal{N}}", from=1-2, to=2-2]
	\arrow["{\mathcal{M}_E}", from=2-1, to=2-2]
  \end{tikzcd}\end{equation}
  commutes, where
  \begin{align*}
    \mathcal{N}(u,\eta) \coloneqq &\ (-u,-\eta),\qquad \mathcal{I}_E
                                \coloneqq \mathcal{D}_{p_0^{-2}}\circ
                     R_{-\pi/2}\circ \iota^* \circ
                                R_{-\pi/2},\\
    \iota(x) \coloneqq &\ \frac{x}{|x|^2},\qquad \iota^*(x,\xi) \coloneqq \Big(\frac{x}{|x|^2},|x|^2\xi-2(x\cdot \xi)x\Big).
\end{align*}
Explicitly,
\begin{equation*}
\mathcal{I}_E(x,\xi)=\Big(p_0^{-2}\big(-|\xi|^2x+2(x\cdot \xi)\xi \big) ,-p_0^2\frac{\xi}{|\xi|^2}\Big).
\end{equation*}
It is easy to see from \eqref{eq:newdiagram} that $\mathcal{I}_E$ is an involution and it takes the set $\Sigma_E \setminus \{(x,0) \colon
|x|=2p_0^{-2}\}$ to itself. Now we define the compactification of $\Sigma_E$:
\begin{equation}\label{eq:compact}
\overline{\Sigma_E} \coloneqq (\Sigma_E^{(0)} \sqcup \Sigma_E^{(1)})/\sim, \qquad  \big(\mathcal{I}_E(x,\xi),0\big)\sim\big((x,\xi),1\big)\text{ for } \xi \neq 0.
\end{equation}

\begin{defn}\label{def:moser2}
For $E<0$, define the regularized Moser map $\overline{\mathcal{M}_E}:\overline{\Sigma_E}\to
T_1^*\mathbb{S}^3$ by
\begin{equation}\label{eq:compactmoser}
\begin{aligned}
\overline{\mathcal{M}_E}\big( (x,\xi),0 \big) &\coloneqq
                                                \mathcal{M}_E(x,\xi),\\
  \overline{\mathcal{M}_E}\big( (x,\xi),1 \big) &\coloneqq
  \mathcal{M}_E\big(\mathcal{I}_E(x,\xi)\big),\text{ when }\xi \neq 0,\\
  \overline{\mathcal{M}_E}\big((x,0),1\big) &\coloneqq
  \big(\mathsf{NP},(2^{-1}p_0^2x,0)\big).
  \end{aligned}
  \end{equation}
  \end{defn}
  \begin{rem}\label{rem:moser2}
One can show $\overline{\mathcal{M}_E}$ is a smooth diffeomorphism, and we
can then define the regularized Hamiltonian flow on
$\overline{\Sigma_E}$. Indeed, for any $t \in \mathbb{R}$,
define $\overline{\Xi_H^t}:\overline{\Sigma_E}\to
\overline{\Sigma_E}$ by
\begin{equation}\label{eq:reghamil}
\overline{\Xi_H^t} \coloneqq \overline{\mathcal{M}_E}^{-1}\circ \Phi_{\mathbb{S}^3}^{s(t)}\circ  \overline{\mathcal{M}_E},
\end{equation}
where $\Phi_{\mathbb{S}^3}^{\bullet}$ denotes the cogeodesic
flow on $T_1^*\mathbb{S}^3$ and $s(t)$ is the inverse of $t(s)$
defined in \eqref{eq:timechange}. Define the inclusion
\begin{equation}\label{eq:inclusion}
i_{\Sigma_E}:\Sigma_E \to \overline{\Sigma_E},\quad i_{\Sigma_E}(x,\xi)=\big((x,\xi),0\big),
\end{equation}
If $(x,\xi) \in \Sigma_E$ is
on a non-collision orbit,
it is easy to see from definitions and Theorem \ref{theo:moser} that
$$\overline{\Xi_H^t}\big((x,\xi),0\big)=
i_{\Sigma_E}\big(\Xi_H^t(x,\xi)\big)$$
for any $t \in \mathbb{R}$, where $\Xi_H^t$ is the
non-regularized Hamiltonian flow.
\end{rem}
\begin{rem}\label{rem:bars}
On the other hand, if
$a \in C_c(\Sigma_E)$, we can extend it to a continuous function $\overline{a}\in
C(\overline{\Sigma_E})$ defined by
\begin{equation}\label{eq:barr}
\begin{aligned}
\overline{a}\big( (x,\xi),0 \big) &\coloneqq
                                               a(x,\xi),\\
  \overline{a}\big( (x,\xi),1 \big) &\coloneqq
  a\big(\mathcal{I}_E(x,\xi)\big),\text{ when }\xi \neq 0,\\
  \overline{a}\big((x,0),1\big) &\coloneqq
 0.
\end{aligned}
\end{equation}
The function $\overline{a}$ is continuous since $\lim_{\xi \to
  0}\overline{a}((x,\xi),1)=0$, which occurs since
$\mathcal{I}_E(x,\xi)$ eventually leaves the support of $a$ as
$\xi \to 0$.
 \par We finally remark that the space of Kepler orbits on
 $\Sigma_E$, $\mathcal{H}(\Sigma_E) \coloneqq \Sigma_E /\sim$
 where $\sim$ denotes equivalence of points on the same orbit,
 is the \textit{same} as the space of regularized Kepler orbits
 $\mathcal{H}(\overline{\Sigma_E}) \coloneqq
 \overline{\Sigma_E}/\sim$. Indeed, these two spaces correspond
 under Moser's regularization to the spaces
 $T_1^*\mathbb{S}_{\neq \mathsf{NP}}^3$ and $T_1^*\mathbb{S}^3$
 quotiented out by points on the same cogeodesic,
 respectively. These two spaces are the same since
 $(\mathsf{NP},\eta)$ is on the same cogeodesic as
 $(\mathsf{SP},-\eta)$. In other words, it doesn't matter if we
 include or exclude the north pole since we are identifying
 points on the same cogeodesic. 
 \par If $\gamma \in \mathcal{H}(\Sigma_E)$, we define
 $\overline{\gamma} \in \mathcal{H}(\overline{\Sigma_E})$ as the
 regularized Kepler orbit starting at
 $i_{\Sigma_E}(\gamma(0))$. With $ a\in C_c(\Sigma_E)$ and
 $\overline{a}\in C(\overline{\Sigma_E})$ defined in
 \eqref{eq:barr},
\begin{align}
\int_{t_{\gamma}-2\pi/p_0^3}^{t_{\gamma}}a(\gamma(t))dt&=\int_{t_{\gamma}-2\pi/p_0^3}^{t_{\gamma}}\overline{a}(\overline{\gamma}(t))dt=\int_0^{2\pi/p_0^3}\overline{a}(\overline{\gamma}(t))dt\quad
                    \text{for
                                                                 }
                                                                 \gamma
                                                                 \text{
                                                                 a
                                                                 collision
                                                                 orbit},\label{eq:seelater}\\
  \int_{0}^{2\pi/p_0^3}a(\gamma(t))dt&=\int_0^{2\pi/p_0^3}\overline{a}(\overline{\gamma}(t))dt\quad \text{for } \gamma \text{ a noncollision orbit},\nonumber                                                                 
\end{align}
where $t_{\gamma}$ is the collision time of $\gamma$ (defined in Remark \ref{rem:collision}).

\end{rem}
\subsubsection{The Quantum Mechanical Fock Map} In this section,
we define the Fock map, first defined by Fock in \cite{F35} (see
\cite{BI66,RC21} for overviews). The Fock map is the
`quantization' of the Moser map. For every
$\hbar>0,N=0,1,2,\ldots$, we define the eigenspace
\begin{equation}\label{eq:eigenspace}
\mathcal{E}_{\widehat{H}_{\hbar}}(\hbar,N) \coloneqq \{ \psi \in H^2(\mathbb{R}^3) \mid \widehat{H}_{\hbar}\psi=E_N(\hbar)\psi\},
\end{equation}
where $\widehat{H}_{\hbar}$ and $E_N(\hbar)$ are defined in
\eqref{eq:operator} and \eqref{eq:energy}, respectively. A priori, elliptic regularity gives $\mathcal{E}_{\widehat{H}_{\hbar}}(\hbar,N) \subset C^{\infty}(\mathbb{R}^3 \setminus \{0\})$. As
noted before, the dimension of $\mathcal{E}_{\widehat{H}_{\hbar}}(\hbar,N)$ is
$(N+1)^2$, and a basis can be found by writing
$\widehat{H}_{\hbar}$ in polar coordinates and separating
the variables $r \geq 0$ and $\theta \in \mathbb{S}^2$
(see, for example, \cite[Theorem 18.3]{H13}). Explicitly, a
basis is given by
\begin{equation}\label{eq:basis}
\psi_{\hbar,N,\ell}^{m}(x) \coloneqq
C_{\hbar,N,\ell}e^{-\frac{1}{\hbar^2(N+1)}|x|}|x|^{\ell}L_{N-\ell}^{(2\ell+1)}\Big(\frac{2}{\hbar(N+1)^2}|x|
\Big)Y_{\ell}^m(\widehat{x}),
\end{equation}
where $\widehat{x}\coloneqq \frac{x}{|x|},\ell \in \{0,\ldots,N\},m \in \{-\ell,\ldots,\ell\},$
$C_{\hbar,N,\ell}$ is a normalization constant to make $\lVert
\psi_{\hbar,N,\ell}^m \rVert_{L^2}=1$, $L_{N-\ell}^{(\bullet)}$
are the generalized Laguerre polynomials of degree $N-\ell$, and
$Y_{\ell}^m$ are the spherical harmonics on $\mathbb{S}^2$ of
degree $\ell$ and order $m$. In particular,
\begin{equation}\label{eq:refagain}
\mathcal{E}_{\widehat{H}_{\hbar}}(\hbar,N) \subset \big(C_c(\mathbb{R}^3) +\mathcal{S}(\mathbb{R}^3)\big) \cap C^{\infty}(\mathbb{R}^3 \setminus \{0\}).
\end{equation}
\par Before we define the Fock map, we analyze $\mathcal{E}_{\widehat{H}_{\hbar}}(\hbar,N)$ in Fourier space. By \eqref{eq:refagain}, the Fourier transform of $\mathcal{E}_{\widehat{H}_{\hbar}}(\hbar,N)$ is contained in $C^{\infty} \cap L^2$. For every $\psi \in \mathcal{E}_{\widehat{H}_{\hbar}}(\hbar,N)$ and for any $\xi \in \mathbb{R}^3$ \begin{align}\label{eq:closeha}
\Big(\frac{|\xi|^2}{2}+\frac{1}{2\hbar^2(N+1)^2} \Big) \mathcal{F}_{\hbar}[\psi](\xi)=\frac{1}{2\pi^2\hbar}\int_{\mathbb{R}^3}\frac{\mathcal{F}_{\hbar}[\psi](p)}{|p-\xi|^2}dp,
\end{align}
where $\mathcal{F}_{\hbar}[\psi](\xi)\coloneqq (2\pi
\hbar)^{-3/2}\int_{\mathbb{R}^3}\psi(v)e^{-i \frac{v\cdot
    \xi}{\hbar}}dv$ is the semiclassical Fourier
transform. This is because
$\mathcal{F}_{\hbar}[|\bullet|^{-1}]=\frac{1}{\pi}\cdot \frac{\sqrt{2\pi \hbar}}{|\bullet|^2}$ and $\mathcal{F}_{\hbar}[f \cdot
  g]=(2\pi \hbar
  )^{-3/2}\mathcal{F}_{\hbar}[f]*\mathcal{F}_{\hbar}[g]$. Define the dilation operator
\begin{equation}\label{eq:DIE}
 \widehat{\mathcal{D}}_{\frac{1}{\hbar(N+1)}}[f] \coloneqq \Big(\frac{1}{\hbar(N+1)} \Big)^{3/2}f\Big(\frac{\bullet}{\hbar(N+1)}\Big).
\end{equation}
We apply $\widehat{\mathcal{D}}_{\frac{1}{\hbar(N+1)}}$ on both sides
of \eqref{eq:closeha} and see that
\begin{align}
  \frac{|\xi|^2+1}{2}(\widehat{\mathcal{D}}_{\frac{1}{\hbar(N+1)}}\circ
  \mathcal{F}_{\hbar})[\psi](\xi) &=\frac{N+1}{2\pi^2}\int_{\mathbb{R}^3}\frac{(\widehat{\mathcal{D} }_{\frac{1}{\hbar(N+1)}}
\circ
                                   \mathcal{F}_{\hbar})[\psi](p)}{|p-\xi|^2}dp. \label{eq:lilwayne}
\end{align}
With $\omega$ defined in \eqref{eq:stereo}, recall that the pullback of the
Euclidean sphere measure $d\Omega$ under $\omega$ is
\begin{equation}\label{eq:measure}
\omega^*d\Omega=\Big(\frac{2}{|p|^2+1} \Big)^3dp.
\end{equation}
Also recall that stereographic projection distorts distances by
the formula
\begin{equation}\label{eq:distort}
|p-\xi|^2=\frac{(|p|^2+1)(|\xi|^2+1)}{4}|\omega(p)-\omega(\xi)|^2.
\end{equation}
We now perform the change of variables of $\xi=\omega^{-1}(u)$ and
$p=\omega^{-1}(y)$ to \eqref{eq:lilwayne}. By \eqref{eq:measure}, we have for any $u \in \mathbb{S}^3$
\begin{align}
  &\frac{|\omega^{-1}(u)|^2+1}{2}(\widehat{\mathcal{D} }_{\frac{1}{\hbar(N+1)}}\circ
  \mathcal{F}_{\hbar})[\psi](\omega^{-1}(u))\nonumber \\
  &\qquad \qquad =\frac{N+1}{2\pi^2}\int_{\mathbb{S}^3}\frac{( \widehat{\mathcal{D}}_{\frac{1}{\hbar(N+1)}}\circ \mathcal{F}_{\hbar})[\psi](\omega^{-1}(y))}{|\omega^{-1}(u)-\omega^{-1}(y)|^2}\Big(\frac{|\omega^{-1}(y)|^2+1}{2}\Big)^3d\Omega(y), \nonumber
\end{align}
which, by \eqref{eq:distort}, implies 
\begin{align}
  &\Big(\frac{|\omega^{-1}(u)|^2+1}{2}\Big)^2(\widehat{\mathcal{D} }_{\frac{1}{\hbar(N+1)}} \circ
    \mathcal{F}_{\hbar})[\psi](\omega^{-1}(u)) \nonumber \\
  &\qquad \qquad =\frac{N+1}{2\pi^2}\int_{\mathbb{S}^3}\frac{(\widehat{\mathcal{D} }_{\frac{1}{\hbar(N+1)}}\circ \mathcal{F}_{\hbar})[\psi](\omega^{-1}(y))}{|u-y|^2}\Big(\frac{|\omega^{-1}(y)|^2+1}{2}\Big)^2d\Omega(y). \label{eq:so}
\end{align}
Define $\mathcal{V}_{\hbar,N}:\mathcal{E}_{\widehat{H}_{\hbar}}(\hbar,N) \to L^2(\mathbb{S}^3)$ by
\begin{equation}\label{eq:phiii}
\mathcal{V}_{\hbar,N}[\psi](u) \coloneqq \Big(\frac{|\omega^{-1}(u)|^2+1}{2}\Big)^2(\widehat{\mathcal{D} }_{\frac{1}{\hbar(N+1)}} \circ
    \mathcal{F}_{\hbar})[\psi](\omega^{-1}(u)).
\end{equation}
Then
    \eqref{eq:so} reads
    \begin{align}\label{eq:so4in}
\mathcal{V}_{\hbar,N}[\psi](u)=\frac{N+1}{2\pi^2}\int_{\mathbb{S}^3}\frac{\mathcal{V}_{\hbar,N}[\psi](y)}{|u-y|^2}d\Omega(y).
    \end{align}
    Note that \eqref{eq:so4in} reflects $\operatorname{SO}(4)$ symmetry: if
    $\mathcal{V}_{\hbar,N}[\psi]$ satisfies \eqref{eq:so4in}, then so does $y \mapsto
   \mathcal{V}_{\hbar,N}[\psi](A^{-1}y)$ for any $A \in \operatorname{SO}(4)$. In fact, $\psi \mapsto \mathcal{V}_{\hbar,N}[\psi]$ is an isometry on $\mathcal{E}_{N}(\hbar)$. Indeed,
    \begin{align}
\lVert \mathcal{V}_{\hbar,N}[\psi] \rVert_{L^2(\mathbb{S}^3)}^2
      &\overset{\mathclap{\eqref{eq:measure}}}{=}\bigg\lVert\Big(\frac{|\bullet|^2+1}{2}\Big)^{\frac{1}{2}}
        (\widehat{\mathcal{D}}_{\frac{1}{\hbar(N+1)}}\circ
        \mathcal{F}_{\hbar})[\psi]\bigg\rVert_{L^2(\mathbb{R}^3)}^2\nonumber\\
      &=\bigg\lVert
        \widehat{\mathcal{D} }_{\frac{1}{\hbar(N+1)}}\Big[\Big(\frac{\hbar^2(N+1)^2|\bullet|^2+1}{2}\Big)^{\frac{1}{2}}
        \mathcal{F}_{\hbar}[\psi]\Big]\bigg\rVert_{L^2(\mathbb{R}^3)}^2\nonumber\\
      &=\bigg\lVert
        \Big(\frac{\hbar^2(N+1)^2|\bullet|^2+1}{2}\Big)^{\frac{1}{2}}
        \mathcal{F}_{\hbar}(\psi)\bigg\rVert_{L^2(\mathbb{R}^3)}^2\nonumber\\
      &=\hbar^2(N+1)^2\Big\langle
        \Big(\frac{|\bullet|^2}{2}-E_N(\hbar)\Big)\mathcal{F}_{\hbar}[\psi],\mathcal{F}_{\hbar}[\psi]
        \Big\rangle_{L^2(\mathbb{R}^3)}\nonumber\\
        &=\hbar^2(N+1)^2\Big\langle
          \Big(-\frac{\hbar^2}{2}\Delta-E_N(\hbar)\Big)\psi,\psi
          \Big\rangle_{L^2(\mathbb{R}^3)}\nonumber\\
      &=\hbar^2(N+1)^2\big\langle(-\hbar^2\Delta-\tfrac{1}{|\bullet|}-2E_N(\hbar))\psi,\psi
        \big\rangle_{L^2(\mathbb{R}^3)}\label{eq:referee}
    \end{align}
where we have added
$(-\frac{\hbar^2}{2}\Delta-\frac{1}{|x|}-E_N(\hbar))\psi=0$ to
the first slot of the inner product. Splitting off the
$\hbar^2(N+1)^2\langle -2E_N(\hbar)\psi,\psi \rangle=\lVert \psi \rVert_{L^2(\mathbb{R}^3)}^2$ term from \eqref{eq:referee},
\begin{align}\label{eq:unitarityhe}
  \lVert \mathcal{V}_{\hbar,N}[\psi] \rVert_{L^2(\mathbb{S}^3)}^2=\lVert \psi \rVert_{L^2(\mathbb{R}^3)}^2+\hbar^2(N+1)^2\underbrace{\Big\langle\Big(-\hbar^2\Delta-\frac{1}{|\bullet|}\Big)\psi,\psi
        \Big\rangle_{L^2(\mathbb{R}^3)}}_{\eqqcolon \text{err}(\psi)}.
\end{align}
We claim $\text{err}(\psi)=0$. Indeed, the commutator
identities $[r\partial_r,r^{-1}]=-r^{-1}, [r\partial_r,-\Delta]=-2\Delta$ imply
\begin{equation}\label{eq:commutator}
[r\partial_r,\widehat{H}_{\hbar}-E_N(\hbar)]=-\hbar^2\Delta-\frac{1}{r} \quad \text{on }C^{\infty}(\mathbb{R}^3 \setminus 0).
\end{equation}
Substituting \eqref{eq:commutator} into $\text{err}(\psi)$ and
using $\widehat{H}_{\hbar}\psi=E_N(\hbar)\psi$, we see
$$\text{err}(\psi)=\langle
(\widehat{H}_{\hbar}-E_N(\hbar))r\partial_r\psi,\psi\rangle_{L^2(\mathbb{R}^3)},$$
but $r\partial_r \psi \in H^2(\mathbb{R}^3)$ by the form of the basis \eqref{eq:basis} and applying $r \partial_r$ on both sides of the eigenvalue equation. It follows that $\operatorname{err}(\psi)=0$ from the self-adjointness of $\widehat{H}_{\hbar}-E_N(\hbar)$. Altogether, we have the definition:
\begin{defn}\label{def:fock} Fix $\hbar>0,N=0,1,2,\ldots$ and let $\mathcal{E}_{\widehat{H}_{\hbar}}(\hbar,N)$ be the eigenspace of
$\widehat{H}_{\hbar}$ with energy $E_N(\hbar)=-\frac{1}{2\hbar^2(N+1)^2}$ (defined in \eqref{eq:eigenspace}). The Fock map
$\mathcal{V}_{\hbar,N}:\mathcal{E}_{\widehat{H}_{\hbar}}(\hbar,N) \to L^2(\mathbb{S}^3)$ is the linear operator defined by
\begin{equation*}
\mathcal{V}_{\hbar,N}(\psi)(u) \coloneqq \Big(\frac{|\omega^{-1}(u)|^2+1}{2}\Big)^2(\widehat{\mathcal{D} }_{\frac{1}{\hbar(N+1)}} \circ
\mathcal{F}_{\hbar})[\psi](\omega^{-1}(u)),
\end{equation*}
where $\omega, \widehat{\mathcal{D}}_{\frac{1}{\hbar(N+1)}}$ are defined in \eqref{eq:stereo},\eqref{eq:DIE}, respectively.
\end{defn}
From \eqref{eq:unitarityhe}, we see that $\mathcal{V}_{\hbar,N}$
is an $L^2$-isometry. The next theorem shows that it is in fact
unitary on it's range. This can be shown in a multitude of different ways,
including using Green's identities \cite[pp. 333]{BI66} (or, relatedly, with layer potential formulas for the sphere \cite[Chapter 11, (11.35)]{T11}), or a group theoretic approach with Schur's lemma \cite[pp. 285]{RC21}. We give a presentation related to the former using the uniqueness of the Dirichlet problem on the ball.
\begin{theo}[\cite{F35}] The Fock map $\mathcal{V}_{\hbar,N}:\mathcal{E}_{\widehat{H}_{\hbar}}(\hbar,N) \to \mathcal{E}_{\mathbb{S}^3}(N)$ is a unitary map where $\mathcal{E}_{\mathbb{S}^3}(N)$ is the space of spherical harmonics of degree $N$.
\end{theo}
\begin{proof}
We begin by showing the range of $\mathcal{V}_{\hbar,N}$ is $\mathcal{E}_{\mathbb{S}^3}(N)$. Define the Riesz potential-type operator $T:L^2(\mathbb{S}^3) \to L^2(\mathbb{S}^3)$ by
$$T[\Phi](u) \coloneqq
\int_{\mathbb{S}^3}\frac{\Phi(y)}{|y-u|^2}d\Omega(y). $$
One can check that $T$ is bounded by Schur's integral test and
changing variables to $\mathbb{R}^3$ with stereographic
projection (see the formulas \eqref{eq:measure} and
\eqref{eq:distort}). We would like to compute $T$, and it
suffices to compute it on each spherical harmonic on $\mathbb{S}^3$. Following \cite[Chapter 8, \S 4]{T11}, for $x \in B \subset \mathbb{R}^4$ in the open unit ball, we have the
equality
\begin{equation}\label{eq:poisson}
|x|^{\ell}Y_{\ell}^{\boldsymbol{m}}(\widehat{x}) = \frac{1-|x|^2}{|\mathbb{S}^3|}\int_{\mathbb{S}^3}\frac{Y_{\ell}^{\boldsymbol{m}}(y)}{|x-y|^4}d\Omega(y).
\end{equation}
Indeed, the left hand side is a harmonic, homogeneous polynomial
on $\mathbb{R}^4$ of degree $\ell$ and the right hand side is
the Poisson kernel applied to $Y_{\ell}^{\boldsymbol{m}}$, so
both sides solve the unique Dirichlet problem
\[\begin{cases}
    \Delta u=0, & \text{on } B\\
    u=Y_{\ell}^{\boldsymbol{m}} , & \text{on } \partial B=\mathbb{S}^3
    \end{cases}\]
where $u \in C(\overline{B}) \cap C^2(B).$ Setting $t \coloneqq -\log|x|$ and letting $x \neq 0$, we have
\begin{equation}\label{eq:poisson2}
e^{-t(\ell+1)}Y_{\ell}^{\boldsymbol{m}}(\widehat{x}) = \frac{2}{|\mathbb{S}^3|}\sinh(t)\int_{\mathbb{S}^3}\frac{Y_{\ell}^{\boldsymbol{m}}(y)}{\big( 2\cosh t - 2(y \cdot \widehat{x})\big)^{2} }d\Omega(y).
\end{equation}
Equation \eqref{eq:poisson2} is true for any $t>0$ and
$\widehat{x}\in \mathbb{S}^3$, so integrating both sides from
$t$ to $\infty$ gives
$$(\ell+1)^{-1}e^{-t(\ell+1)}Y_{\ell}^{\boldsymbol{m}}(\widehat{x})
=\frac{1}{|\mathbb{S}^3|}\int_{\mathbb{S}^3}\frac{Y_{\ell}^{\boldsymbol{m}}(y)}{
  2\cosh t - 2(y\cdot \widehat{x}) }d\Omega(y),
\quad \text{for all } t>0.$$
Now taking $t \to 0^+$ and applying the dominated convergence
theorem, we recover $T[Y_{\ell}^{\boldsymbol{m}}]$ on the right hand side:
$$T [Y_{\ell}^{\boldsymbol{m}}]=\frac{2\pi^2}{\ell+1}Y_{\ell}^{\boldsymbol{m}}.$$
Since $-\Delta_{\mathbb{S}^3}Y_{\ell}^{\boldsymbol{m}}=\ell(\ell+2)Y_{\ell}^{\boldsymbol{m}}$, we see that $T=2\pi^2(-\Delta_{\mathbb{S}^3}+1)^{-1/2}$. Applying $T^{-1}$ on both sides of \eqref{eq:so4in}, we see the image of $\mathcal{V}_{\hbar,N}$ is in $\mathcal{E}_{\mathbb{S}^3}(N)$. Since $\mathcal{V}_{\hbar,N}$ is an $L^2$-isometry and $\operatorname{dim} \mathcal{E}_{\widehat{H}_{\hbar}}(\hbar,N) = \operatorname{dim}\mathcal{E}_{\mathbb{S}^3}(N)=(N+1)^2$, we see $\mathcal{V}_{\hbar,N}$ is unitary, as desired.
\end{proof}
\begin{rem} For
$\mathcal{V}_{\hbar,N}^{-1}$, it will be useful to write it as a composition of operators
\begin{equation}\label{eq:inverse}
\mathcal{V}_{\hbar,N}^{-1} \coloneqq \widehat{\mathcal{D}}_{\frac{1}{\hbar(N+1)}} \circ \mathcal{F}_{\hbar}^{-1} \circ J^{1/2} \circ K 
\end{equation}
where the $L^2$ isometry $K:\mathcal{E}_{N}(\mathbb{S}^3) \to L^2(\mathbb{R}^3)$ and the
multiplication map $J:L^2(\mathbb{R}^3) \to L^2(\mathbb{R}^3)$
are defined by
\begin{equation}\label{eq:JK}
K(f) \coloneqq \Big(\frac{2}{| \bullet |^2+1}\Big)^{3/2}f \circ \omega \quad \text{and} \quad J(f) \coloneqq \frac{2}{| \bullet |^2+1}f.
\end{equation}
It is easy to see that $K$ is an $L^2$ isometry by
\eqref{eq:measure}. Note that if we were to define $\mathcal{V}_{\hbar,N}^{-1}$ on the larger space $L^2(\mathbb{S}^3)$, it would
fail to be unitary due to the $J$
operator, but it \textit{is} unitary as an operator defined on
$\mathcal{E}_{\mathbb{S}^3}(N)$.
\par For $g \in \operatorname{SO}(3)$, define
$$\rho_{\operatorname{SO}(3)}(g):L^2(\mathbb{R}^3) \to L^2(\mathbb{R}^3),\quad \rho_{\operatorname{SO}(3)}(g)[f] \coloneqq f(g^{-1}\bullet),$$
and, for $\widetilde{g} \in \operatorname{SO}(4)$, define
$$\rho_{\operatorname{SO}(4)}(\widetilde{g}):L^2(\mathbb{S}^3)
\to L^2(\mathbb{S}^3),\quad
\rho_{\operatorname{SO}(4)}(\widetilde{g})[f] \coloneqq
f(\widetilde{g}^{-1}\bullet)$$
by rotation. For any $g \in \operatorname{SO}(3)$, one can show
\begin{equation}\label{eq:kk}
\mathcal{V}_{\hbar,N}^{-1}\circ \rho_{\operatorname{SO}(4)}\big(\begin{psmallmatrix}g & 0\\ 0& 1 \end{psmallmatrix}\big)=\rho_{\operatorname{SO}(3)}(g)\circ \mathcal{V}_{\hbar,N}^{-1}.
\end{equation}
Indeed, this follows from the invariance of $\widehat{\mathcal{D}}_{\frac{1}{\hbar(N+1)}},\mathcal{F}_{\hbar},J$ under rotations, and the fact that rotations in $\mathbb{R}^3$ transform to rotations fixing the north pole on $\mathbb{S}^3$ under stereographic projection.

\end{rem}

\section{Proof of Theorem \ref{theo:1}}
We start by defining the Coulomb coherent states. Following \cite{U84}, \cite[Appendix 1]{HV-B12}, \cite[Appendix 2]{A-CHV-B17}, \cite[Chapter 9, \S9.3]{RC21}, we define the set
\begin{equation*}
\mathcal{A} \coloneqq \{ \alpha \in \mathbb{C}^4\ ; \ |\Re \alpha|=|\Im \alpha|=1, \Re (\alpha) \cdot \Im (\alpha) =0\}.
\end{equation*}
Note that $\mathcal{A}$ is a parametrization $T_1\mathbb{S}^3$ (and hence $T_1^*\mathbb{S}^3$) where $\Re \alpha\in \mathbb{S}^3$ is the position vector and $\Im \alpha\in \mathbb{S}^3$ is the velocity vector. Recall the highest weight spherical harmonics (also called spherical coherent states) $\Phi_{\alpha,N}\in L^2(\mathbb{S}^3)$ are defined by
$$\Phi_{\alpha,N}(u)\coloneqq c_N(\alpha \cdot u)^N,$$
for any $\alpha \in \mathcal{A}$ where $c_N \coloneqq \frac{1}{\pi \sqrt{2}}\sqrt{N+1}$ is a normalization constant so that $\lVert \Phi_{\alpha,N}\rVert_{L^2(\mathbb{S}^3)}=1$. It is well-known that as $N \to \infty$, $\Phi_{\alpha,N}$
concentrates on the great circle $\{u \in \mathbb{S}^3: |\alpha \cdot u|=1\}$ (see \cite{TV-B97,JZ99}). Now
we define the Coulomb coherent states.
\begin{defn}[Coulomb coherent states]\label{def:fergie}
Fix $E<0$ and let $\hbar>0,N=0,1,2,\ldots,$ be such that $E_N(\hbar)=E$. If $\overline{\gamma}$ is a regularized Kepler orbit on $\overline{\Sigma_E}$, we define $\Psi_{\hbar,N}^{\gamma} \in
L^2(\mathbb{R}^3)$ by
\begin{equation}\label{eq:hcs}
\Psi_{\hbar,N}^{\gamma}\coloneqq \mathcal{V}_{\hbar,N}^{-1}(\Phi_{\alpha_{\gamma},N}),
\end{equation}
where $\alpha_{\gamma}\coloneqq \overline{\mathcal{M}_E}^{-1}(\overline{\gamma}(0))\in \mathcal{A}$, $\overline{\mathcal{M}_E}$ is the regularized Moser map (defined in Definition \ref{def:moser2}), and $\mathcal{V}_{\hbar,N}^{-1}$ is the inverse of the Fock map (defined in Definition \ref{def:fock}, and again in \eqref{eq:inverse}).
\end{defn}
\begin{rem}
We briefly note that this definition is projective in the sense
that if $\overline{\gamma}$ is the same Kepler orbit with a
different initial point, $\Psi_{\hbar,N}^{\gamma}$ will be the
same up to a constant phase factor. This is because a different
initial point is equivalent to rotating $\alpha$ in the $(\Re
\alpha, \Im \alpha)$-plane (i.e. $\alpha \mapsto
e^{i\theta}\alpha$), thus changing $\Phi_{\alpha,N}$ by a
constant phase factor and, consequently, changing
$\Psi_{\alpha,N}^{\gamma}$ by the same factor.\par
From Remark \ref{rem:bars}, we recall that the Kepler orbits
$\gamma$ are in one-to-one correspondence with the regularized
Kepler orbits $\overline{\gamma}$. It is for this reason that we
elect for the less notationally heavy $\Psi_{\hbar,N}^{\gamma}$
rather than $\Psi_{\hbar,N}^{\overline{\gamma}}$. In fact, in lieu of the previous paragraph, we can assume $\overline{\gamma}(0)\in \Sigma_E$. In this case, we can define $\Psi_{\hbar,N}^{\gamma}$ with the unregularized Moser map $\mathcal{M}_E$ in the same way and obtain the same (projective) definition.
\end{rem}
We begin with an argument using rotation symmetry so that we may assume, without loss of generality, that $\alpha_{\gamma}=\alpha(\theta_0) \coloneqq e_1+i(\cos
  (\theta_0)e_2+\sin(\theta_0)e_4)$ for some $\theta_0 \in [0,2\pi)$. Geometrically, this corresponds to the great circle $$\big\{e_1\cos s+(e_2\cos
  \theta_0+e_4\sin\theta_0)\sin s \mid s \in [0,2\pi)\big\}.$$ This reduction was done in \cite[(4.56)]{TV-B97}, and we give more details here.
\proofpart{Step}{Reduction to $\alpha_{\gamma}=\alpha(\theta_0) \coloneqq e_1+i(\cos
  (\theta_0)e_2+\sin(\theta_0)e_4)$}
We claim that if the result is true for $\gamma_0$ such that $\alpha_{\gamma_0} = e_1+i(\cos
  (\theta_0)e_2+\sin(\theta_0)e_4)$, then it is also true for any $\alpha_{\gamma} \in \mathcal{A}$. Indeed, let $\overline{\gamma}$ be a regularized Kepler orbit and $\varphi(s)=\Re \alpha_{\gamma} \cos s + \Im \alpha_{\gamma} \sin s$ be the corresponding great circle on $\mathbb{S}^3$. There exists an $s_0$ such that the fourth coordinate of $\varphi(s_0)$ is zero (since $a \cos s +b \sin s$ can be written as a single trigonometric function with a different amplitude and shifted phase). By reparametrizing $\varphi$ to begin at $s_0$, we can assume the fourth coordinate of $\Re \alpha_{\gamma}$ is zero. There exists a rotation in the first three coordinates of $\varphi$ such that the initial point of $\varphi$ is at $e_1$. That is, there exists $g \in \operatorname{SO}(3)$ such that
\begin{equation*}\label{eq:wlog}
\begin{pmatrix}g & 0\\ 0& 1 \end{pmatrix}\varphi(s)= e_1 \cos s +(a_2e_2+a_3e_3+a_4e_4)\sin s.
\end{equation*}
where $e_j$ is the $j$th standard basis vector in $\mathbb{R}^4$, and $a_j \in \mathbb{R}$ are such that $a_2^2+a_3^2+a_4^2=1$. We can apply a further rotation in the $e_2e_3$-plane so as to make $a_3=0$, so altogether there exists $g \in \operatorname{SO}(3)$ such that
\begin{equation}\label{eq:wlog}
\begin{pmatrix}g & 0\\ 0& 1 \end{pmatrix}\varphi(s)= e_1 \cos s +(\cos (\theta_0)e_2+\sin(\theta_0)e_4)\sin s.
\end{equation}
for some $\theta_0 \in [0,2\pi)$. That is, $\begin{pmatrix}g & 0\\ 0& 1 \end{pmatrix}\alpha_{\gamma}=\alpha_{\gamma_0}=\alpha(\theta_0)$. Then 
\begin{align}
  \langle
  \operatorname{Op}_{\hbar}(a)\Psi_{\hbar,N}^{\gamma_0},\Psi_{\hbar,N}^{\gamma_0}
  \rangle&=\langle
  \operatorname{Op}_{\hbar}(a)\mathcal{V}_{\hbar,N}^{-1}[\Phi_{\alpha_{\gamma_0},N}],\mathcal{V}_{\hbar,N}^{-1}[\Phi_{\alpha_{\gamma_0},N}]
           \rangle\nonumber \\
  &=\big\langle
  \operatorname{Op}_{\hbar}(a)\big(\mathcal{V}_{\hbar,N}^{-1}\circ \rho_{\operatorname{SO}(4)}\big( \begin{psmallmatrix}g & 0\\ 0& 1 \end{psmallmatrix}\big)\big)[\Phi_{\alpha_{\gamma},N}],\big(\mathcal{V}_{\hbar,N}^{-1}\circ \rho_{\operatorname{SO}(4)}\big( \begin{psmallmatrix}g & 0\\ 0& 1 \end{psmallmatrix}\big)\big)[\Phi_{\alpha_{\gamma},N}]
  \big\rangle\nonumber \\
&\overset{\mathclap{\eqref{eq:kk}}}{=}\langle
  \operatorname{Op}_{\hbar}(a)(\rho_{\operatorname{SO}(3)}(g) \circ\mathcal{V}_{\hbar,N}^{-1})[\Phi_{\alpha_{\gamma},N}],(\rho_{\operatorname{SO}(3)}(g) \circ\mathcal{V}_{\hbar,N}^{-1})[\Phi_{\alpha_{\gamma},N}]
  \rangle\nonumber \\
  &=\big\langle
   \rho_{\operatorname{SO}(3)}(g)\big[\operatorname{Op}_{\hbar}(a(g\bullet ,g\bullet ))[\Psi_{\hbar,N}^{\gamma}]\big],\rho_{\operatorname{SO}(3)}(g)[\Psi_{\hbar,N}^{\gamma}]
    \big\rangle\nonumber \\
    &=\langle \operatorname{Op}_{\hbar}(a(g\bullet ,g\bullet))[\Psi_{\hbar,N}^{\gamma}],[\Psi_{\hbar,N}^{\gamma}] \rangle. \label{eq:eqnn}
\end{align}
Using the shorthand $\lim_{\hbar,N}$ for the limit as $\hbar \to 0,N \to \infty$ with $E_N(\hbar) = E,$
\begin{align*}
\lim_{\hbar,N}\langle \operatorname{Op}_{\hbar}(a(g\bullet ,g\bullet))\Psi_{\hbar,N}^{\gamma},\Psi_{\hbar,N}^{\gamma} \rangle \overset{\mathclap{\eqref{eq:eqnn}}}{=}\lim_{\hbar,N}  \langle
  \operatorname{Op}_{\hbar}(a)\Psi_{\hbar,N}^{\gamma_0},\Psi_{\hbar,N}^{\gamma_0}
  \rangle =\int_{\gamma_0}a\overset{\mathclap{\eqref{eq:later}}}{=}\int_{\gamma}a(g \bullet,g\bullet),
\end{align*}
as desired.
\par Before we move to the second step, we recall from Remark \ref{rem:collision} that the non-collision orbits correspond to great circles not going through the north pole. That is, when $\theta_0 \neq \pi/2,3\pi/2$ in our reduced $\alpha(\theta_0)$.
\proofpart{Step}{$\gamma$ is not a collision orbit (i.e. $\theta_0 \neq \pi/2,3\pi/2$)}
We prove the theorem for $\gamma$ not being a collision orbit, which will be important to the statement of Lemma \ref{lem:crit}. Suppose $\alpha=\alpha(\theta_0)$, defined in the statement of the previous step. Since $E_{N}(\hbar)=E$, we again use the notation $p_0=\sqrt{-2E}=\frac{1}{\hbar(N+1)}$. If $a \in C_c^{\infty}(T^*\mathbb{R}^3)$, then we have
$$\langle \operatorname{Op}_{\hbar}(a)\Psi_{\hbar,N}^{\gamma},\Psi_{\hbar,N}^{\gamma}\rangle=\int_{T^*\mathbb{R}^3}a(x,\xi)W_{\Psi_{\hbar,N}^{\gamma}}(x,\xi)dxd\xi,$$
where $W_{\Psi_{\hbar,N}^{\gamma}} \in C_0(T^*\mathbb{R}^3)\cap L^2(T^*\mathbb{R}^3)$ (see \cite[Proposition 1.92]{F89}) is such that
$$W_{\Psi_{\hbar,N}^{\gamma}}(x,\xi) \coloneqq
\frac{1}{(2\pi
  \hbar)^3}\int_{\mathbb{R}^3}\Psi_{\hbar,N}^{\gamma}(x+\tfrac{v}{2})\overline{\Psi_{\hbar,N}^{\gamma}(x-\tfrac{v}{2})}e^{-\frac{i}{\hbar}\langle
  v,\xi \rangle}dv. $$
Using basic facts about Wigner distributions (see \cite[Proposition 1.94]{F89}), we see
\begin{align*}
W_{\Psi_{\hbar,N}^{\gamma}}(x,\xi) &=
                                             W_{\mathcal{V}_{\hbar,N}^{-1}[\Phi_{\alpha,N}]}(x,\xi)\\
  &= W_{(\widehat{\mathcal{D} }_{p_0}\circ \mathcal{F}_{\hbar}^{-1} \circ J^{1/2}
    \circ K)[\Phi_{\alpha,N}]}(x,\xi)\\
  &= W_{(\mathcal{F}_{\hbar}^{-1} \circ J^{1/2} \circ
    K)[\Phi_{\alpha,N}]}(p_0x,p_0^{-1}\xi)\\
  &= W_{( J^{1/2} \circ K)[\Phi_{\alpha,N}]}(p_0^{-1}\xi,-p_0x).
\end{align*}
So we have
\begin{align}
&\langle \operatorname{Op}_{\hbar}(a)\Psi_{\hbar,N}^{\gamma},\Psi_{\hbar,N}^{\gamma}\rangle =
                                           \int_{T^*\mathbb{R}^3}a(x,\xi)W_{\Psi_{\hbar,N}^{\gamma}}(x,\xi)dxd\xi\nonumber
\\
 &=\int_{T^*\mathbb{R}^3}a(p_0^{-1}x,p_0\xi)W_{\Psi_{\hbar,N}^{\gamma}}(p_0^{-1}x,p_0\xi)dxd\xi
  \nonumber\\
    &=\int_{T^*\mathbb{R}^3}a(p_0^{-1}x,p_0\xi)W_{(J^{1/2}\circ K)[\Phi_{\alpha,N}]}(\xi,-x)dxd\xi\nonumber\\
 &=\frac{c_N^2}{(2\pi \hbar)^3}\int_{\mathbb{R}^3}\int_{T^*\mathbb{R}^3} \frac{16 a(p_0^{-1}x,p_0\xi)(\alpha
\cdot \omega(\xi+\frac{v}{2}))^N(\overline{\alpha}
\cdot \omega(\xi-\frac{v}{2}))^Ne^{\frac{i}{\hbar}\langle
    v,x \rangle}}{ (|\xi+\frac{v}{2}|^2 +
               1)^2( |\xi-\frac{v}{2}|^2
               + 1)^2}dxd\xi dv\nonumber \\
    &=\frac{(N+1)^4}{16\pi^5}\int_{\mathbb{R}^3}\int_{T^*\mathbb{R}^3} f(x,\xi,v)e^{iN P(x,\xi,v)}dxd\xi dv, \label{eq:newish2}
\end{align}
where the last line we use $p_0^{-1}=\hbar(N+1)$ and the substitution $x \mapsto p_0^{-1}x$ while defining
\begin{align*}
f(x,\xi,v) &\coloneqq \frac{ 16a(p_0^{-2}x,p_0\xi)e^{i\langle v,x \rangle}}{ (|\xi+\frac{v}{2}|^2 +
               1)^2( |\xi-\frac{v}{2}|^2
               + 1)^2} \\
P(x,\xi,v) &\coloneqq -i\log\big(\alpha
\cdot \omega(\xi+\tfrac{v}{2})\big)-i\log\big(\overline{\alpha}
\cdot \omega(\xi-\tfrac{v}{2})\big)+\langle v,x \rangle.
\end{align*}
First note that $\Im P(x,\xi,v) \geq 0$. This is because $|\alpha
\cdot \omega(\xi\pm \tfrac{v}{2})| \leq 1$ since $|\alpha
\cdot \omega(\xi\pm \tfrac{v}{2})|$ is the norm of
projection of $\omega(\xi\pm \tfrac{v}{2})$ on the
$\operatorname{span}_{\mathbb{R}}(\Re\alpha,\Im\alpha)$. In particular, we have equality if and only if $\omega(\xi\pm \tfrac{v}{2})\in \operatorname{span}_{\mathbb{R}}(\Re\alpha,\Im\alpha)$.
We would like to apply stationary phase methods to formula
\eqref{eq:newish2}. We have the following lemma.
\begin{lem}\label{lem:crit}
For the complex phase $P$ above, let $\mathcal{C} \coloneqq \{\nabla_x P =\nabla_{\xi} P = \nabla_v P=0,\Im P(x,\xi,v)=0\}$ be the critical manifold. Then
$$\mathcal{C}=
\bigg\{(x,\xi,v)=\bigg(\begin{psmallmatrix}\sin \beta -\sin \theta_0 \\ -\cos \theta_0 \cos \beta\\ 0 \end{psmallmatrix},\tfrac{1}{1-\sin
                         \theta_0 \sin \beta}\begin{psmallmatrix}\cos(\beta)\\ \sin(\beta)\cos(\theta_0)\\
                 0 \end{psmallmatrix},\begin{psmallmatrix} 0 \\ 0\\ 0 \end{psmallmatrix}\bigg) \mid \beta \in [0,2\pi)\bigg\}.$$
That is, $\mathcal{C}=\operatorname{image}(\gamma_0)\times \{(0,0,0)\}$ where $\gamma_0$ is the Kepler orbit on $\Sigma_{-1/2}$ that contains $\mathcal{M}_{-1/2}(\alpha(\theta_0))$.          
\end{lem}
\begin{proof}
The condition $\nabla_x P =0$ implies $v=0$. As noted above, the second condition is equivalent to the
condition $\omega(\xi\pm \tfrac{v}{2})\in
\operatorname{span}_{\mathbb{R}}(\Re\alpha,\Im\alpha)$. Let $\beta$ be such that $\alpha \cdot \omega(\xi)=e^{i\beta}$. Since $\alpha = e_1+i(\cos(\theta_0)e_2+\sin(\theta_0)e_4)$, we have
$$\omega(\xi)=\cos(\beta)e_1+\sin(\beta)(e_2\cos
\theta_0+ e_4\sin \theta_0).$$
 Taking $\omega^{-1}$ on both sides, we have
\begin{align*}
\xi_1 =\frac{\cos(\beta)}{1-\sin
                         \theta_0 \sin \beta},\quad \xi_2=\frac{\sin(\beta)\cos(\theta_0)}{1-\sin\theta_0 \sin \beta},\quad \xi_3=0.
\end{align*}
Finally, the $\partial_{v_j}P(x,\xi,0)=0$
reads
\begin{align*}
-i\frac{\alpha_j+[\alpha_4-\alpha \cdot \omega(\xi)]\xi_j}{(|\xi|^2+1)(\alpha \cdot \omega(\xi))}+i\frac{\overline{\alpha}_j+[\overline{\alpha}_4-\overline{\alpha} \cdot \omega(\xi)]\xi_j}{(|\xi|^2+1)(\overline{\alpha} \cdot \omega(\xi))}+x_j=0,
\end{align*}
which implies
\begin{align*}
x_j&=  \Re \bigg( i\frac{2\alpha_j}{(|\xi|^2+1)(\alpha \cdot
     \omega(\xi))}+i\omega(\xi)_j\frac{\alpha_4-\alpha \cdot
     \omega(\xi)}{(\alpha \cdot \omega(\xi))}\bigg)\\
  &=(1-\sin \theta_0 \sin \beta)\Re \big(i\alpha_je^{-i\beta}\big)-\omega(\xi)_j\cos \beta \sin \theta_0.
\end{align*}
We see that
\begin{align*}
  x_1 = \sin \beta-\sin\theta_0,\quad x_2 = -\cos \beta\cos \theta_0,\quad x_3=0,
\end{align*}
as desired.
\end{proof}
Let $\pi_x \mathcal{C}$ denote the projection of $\mathcal{C}$
to configuration space, and let $\chi \in
C_c^{\infty}(\mathbb{R}^3,[0,1])$ be a smooth bump
function that is $1$ on $\pi_x \mathcal{C}$ and $0$ off of a
small tubular neighborhood of $\pi_x \mathcal{C}$. Then the integral in \eqref{eq:newish2} becomes
\begin{align}
\int_{\mathbb{R}^3}\int_{T^*\mathbb{R}^3}
  f(x,\xi,v)e^{iN P(x,\xi,v)}dxd\xi
  dv&=\int_{\mathbb{R}^3}\int_{T^*\mathbb{R}^3}\chi(x)
      f(x,\xi,v)e^{iN P(x,\xi,v)}dxd\xi dv \nonumber \\
  &\qquad +\int_{\mathbb{R}^3}\int_{T^*\mathbb{R}^3}(1-\chi(x)) f(x,\xi,v)e^{iN P(x,\xi,v)}dxd\xi dv. \label{eq:newish3}
\end{align}
We claim the second integral of \eqref{eq:newish3} is
$O(N^{-\infty})$. Indeed, if we further split the integral with a smooth
bump function in $v$ with support in a neighborhood of the
origin, we see that the integral for $v$ small is covered by the
method of nonstationary phase \cite[Theorem 7.7.1]{H85} since the support is outside $\mathcal{C}$. For $v$ large, we observe that
$$\frac{1}{iNv_j}\partial_{x_j}e^{iNP(x,\xi,v)}=e^{iNP(x,\xi,v)}.$$
So we can repeatedly apply integration by parts and gain powers of $N$ in the denominator since $v_j^{-1}\partial_{x_j}[(1-\chi)f]=O(|v|^{-8})$.
\par For the first integral of \eqref{eq:newish3}, we apply the
change of variables $x \mapsto (\beta,t,s)$ where
$$x=x(\beta)+tn_{\beta}+se_3 \quad \text{where} \quad x(\beta)
\coloneqq \begin{psmallmatrix}\sin \beta -\sin \theta_0 \\ -\cos
            \theta_0 \cos \beta\\ 0 \end{psmallmatrix},\ \ 
            n_{\beta} \coloneqq \tfrac{1}{\sqrt{1-\sin^2\beta
                \sin^2\theta_0}}\begin{psmallmatrix}
                        -\sin\beta \cos \theta_0\\ \cos \beta\\
                                  0 \end{psmallmatrix},$$
where $\beta \in [0,2\pi)$ and $t^2+s^2<\delta$ for some
$\delta>0$. Geometrically, $x(\beta)$ is the point along the orbit in configuration space, $n_{\beta}$ is the unit normal vector orthogonal to the $(e_1,e_2)$-plane containing the configuration space orbit, and $e_3$ is the unit normal vector to this plane. That is, $\lVert n_{\beta} \rVert =1$ and $x(\beta) \cdot n_{\beta}=0$, so the change of variables parametrizes a tubular neighborhood of $\pi_x \mathcal{C}$. With this change of variables, it can be computed that $$dx=\Big|\sqrt{1-\sin^2\beta
                \sin^2\theta_0}+t \cos \theta_0\Big|dtdsd\beta.$$ Observe that the Jacobian factor is smooth and non-vanishing close enough to $\pi_x \mathcal{C}$. Altogether, by \eqref{eq:newish3}, we have
\begin{align}
\int_{\mathbb{R}^3}\int_{T^*\mathbb{R}^3}
  \chi(x)f(x,\xi,v)e^{iN P(x,\xi,v)}dxd\xi=\int_0^{2\pi}\int_{t^2+s^2<\delta}\int_{\mathbb{R}^6}\widetilde{f}_{\beta}(t,s,\xi,v)e^{iN \widetilde{P}_{\beta}(t,s,\xi,v)}d\xi dv dtds d\beta,\label{eq:barely}
\end{align}
where
\begin{align*}
\widetilde{f}_{\beta}(t,s,\xi,v)\coloneqq&\ f(x(\beta)+tn_{\beta}+se_3,\xi,v)\chi(x(\beta)+tn_{\beta}+se_3)\Big|\sqrt{1-\sin^2\beta
                                         \sin^2\theta_0}+t \cos \theta_0\Big|\\
\widetilde{P}_{\beta}( t,s,\xi,v)  \coloneqq& \  P(x(\beta)+tn_{\beta}+se_3,\xi,v).
\end{align*}
For fixed $\beta$, we apply the method of stationary phase in the variables $(t,s,\xi,v)$. By Lemma \ref{lem:crit}, the only critical point of $\widetilde{P}$ is at $(0,0,\xi(\beta),0)$ where $\xi(\beta) \coloneqq \tfrac{1}{1-\sin
                         \theta_0 \sin
                         \beta}\begin{psmallmatrix}\cos(\beta)\\
                                 \sin(\beta)\cos(\theta_0)\\
  0 \end{psmallmatrix}$. The
                                 Hessian of $\widetilde{P}$
                                 evaluated at this critical
                                 point is
                                 \begin{equation}\label{eq:HESS}
                                 \operatorname{Hess}(\widetilde{P})_{crit}=\bordermatrix{     
            & t     & s   & \xi & v      \cr
    t     & 0     & 0    & 0 & n_{\beta}^T      \cr
    s    & 0    & 0     & 0 & e_3^T      \cr
    \xi  & 0  & 0  & -2i \Re H_{\beta} & \Im H_{\beta}   \cr
    v     & n_{\beta}     & e_3     & \Im H_{\beta} & -\tfrac{i}{2}\Re H_{\beta}      \cr
  }, \end{equation}
  where $H_{\beta}$ is the Hessian of the function $\xi \mapsto
  \log(\alpha \cdot \omega(\xi))$ evaluated at $\xi(\beta)$. The
  calculation of the determinant of the matrix $H_{\beta}$ was
  studied in \cite[Appendix A.3]{RC21}. Using a lower triangular block matrix
  identity, we have
\begin{equation}\label{eq:determinant}
\det \operatorname{Hess}(\widetilde{P})_{\beta}=-\det \begin{pmatrix}-2i \Re H_{\beta} & \Im H_{\beta}   \cr
   \Im H_{\beta} & -\tfrac{i}{2}\Re H_{\beta} \end{pmatrix} \det \begin{pmatrix}0 & n_{\beta}^T\\ 0 & e_3^T \end{pmatrix}\begin{pmatrix}-2i \Re H_{\beta} & \Im H_{\beta}   \cr
   \Im H_{\beta} & -\tfrac{i}{2}\Re H_{\beta} \end{pmatrix}^{-1}\begin{pmatrix}0 & 0\\ n_{\beta} & e_3 \end{pmatrix}.
\end{equation}
  Using the block matrix
  identity
  \begin{equation}\label{eq:blockid}
\begin{pmatrix}\frac{1}{2}I & -I \\ 0 &
                                          I \end{pmatrix}\begin{pmatrix}2A & B \\ B & \frac{1}{2}A \end{pmatrix}\begin{pmatrix}I &  \frac{1}{2}I\\ 0 & I \end{pmatrix}=\begin{pmatrix}A-B &  0\\ B &\frac{1}{2}(A+B) \end{pmatrix},
                                          \end{equation}
                                          we see
\begin{equation}\label{eq:hellomatrix}
\Big|\det \begin{pmatrix}-2i \Re H_{\beta} & \Im H_{\beta}   \cr
   \Im H_{\beta} & -\tfrac{i}{2}\Re H_{\beta} \end{pmatrix}\Big|=|\det H_{\beta}|^2.
   \end{equation}
Inverting the identity \eqref{eq:blockid} and using the formula
for the inverse of a triangular block matrix, we see
$$\begin{pmatrix}-2i \Re H_{\beta} & \Im H_{\beta}   \cr
   \Im H_{\beta} & -\tfrac{i}{2}\Re
                   H_{\beta} \end{pmatrix}^{-1}=\begin{pmatrix}
                  * & * \\ * & 2
                                                               \Re(H_{\beta}^{-1})\end{pmatrix}. $$
But note $n_{\beta}$ and $e_3$ are eigenvectors of $
\Re(H_{\beta}^{-1})$ with eigenvalues $\lambda_2,\lambda_3$
defined in \cite[(9.39)]{RC21}. Thus
\begin{equation}\label{eq:hessfinal}
\sqrt{|\det \operatorname{Hess}(\widetilde{P})_{\beta}|}\overset{\eqref{eq:determinant},\eqref{eq:hellomatrix}}{=}2|\det H_{\beta}| \sqrt{\lambda_2\lambda_3}=2(1-\sin \beta \sin \theta_0)^3\sqrt{1-\sin^2 \beta \sin^2 \theta_0},
\end{equation}
where the last equality follows from (9.33) and (9.39) in
Chapter 9 of \cite{RC21} ($n_{\beta}$ is the normalized
$v_{\beta}$ in \cite{RC21}). Now we apply stationary phase to
\eqref{eq:barely}, and with \eqref{eq:newish3}, we see
\begin{align*}
&\frac{(N+1)^4}{16\pi^5}\int_{\mathbb{R}^3}\int_{T^*\mathbb{R}^3}
  f(x,\xi,v)e^{iN P(x,\xi,v)}dxd\xi dv\\
  &=\frac{(N+1)^4}{16\pi^5}\Big(\frac{2\pi}{N}\Big)^4\int_{0}^{2\pi}\frac{16a(p_0^{-2}x(\beta),p_0
    \xi(\beta))}{(|\xi(\beta)|^2+1)^4}\frac{\sqrt{1-\sin^2\beta
    \sin^2\theta_0}}{2(1-\sin \beta \sin
    \theta_0)^3\sqrt{1-\sin^2 \beta \sin^2
    \theta_0}}d\beta+O(\tfrac{1}{N}) \\
  &=\frac{1}{2\pi} \int_{0}^{2\pi}a(p_0^{-2}x(\beta),p_0
    \xi(\beta))(1-\sin \beta \sin
    \theta_0)d\beta+O(\tfrac{1}{N})\\
  &=\frac{p_0^3}{2\pi} \int_{0}^{2\pi/p_0^3}a\big(\gamma(t)\big)dt+O(\tfrac{1}{N}),
\end{align*}
where the last line we change variables $\beta \to t$ where $t$
is as in Theorem \ref{theo:moser}.
\stepcounter{proofpart}\stepcounter{proofpart}
\proofpart{Step}{$\gamma$ is a collision orbit
  (i.e. $\theta_0=\pi/2,3\pi/2)$}
By reversing time, we can assume without loss of generality that
$\theta_0=\pi/2$. The setup is the same as in Step 2. We still
consider the integral \eqref{eq:newish2}, but the critical
manifold is now
\begin{equation*}
\mathcal{C} =\bigg\{\bigg(\begin{psmallmatrix}\sin \beta-1 \\ 0\\ 0 \end{psmallmatrix},\tfrac{1}{1-\sin \beta}\begin{psmallmatrix}\cos \beta\\ 0\\
                 0 \end{psmallmatrix},\begin{psmallmatrix} 0 \\ 0\\ 0 \end{psmallmatrix}\bigg) \mid \beta \in (-3\pi/2,\pi/2)\bigg\}.
                 \end{equation*}
We cannot apply the same change of variables in only the $x$
variables as before since the manifold degenerates into a line segment
when projected to configuration space. We instead consider a tubular neighborhood of $\mathcal{C} \cap \operatorname{supp} a(p_0^{-2}\bullet, p_0 \bullet)$ in phase space. Let $\chi \in C_{c}^{\infty}(T^*\mathbb{R}^3,[0,1])$ be a smooth bump function that is $1$ on $\mathcal{C} \cap \operatorname{supp} a(p_0^{-2}\bullet, p_0 \bullet)$ and $0$ off of a
small tubular neighborhood of $\mathcal{C} \cap \operatorname{supp} a(p_0^{-2}\bullet, p_0 \bullet)$. Then we have
\begin{align}
\int_{\mathbb{R}^3}\int_{T^*\mathbb{R}^3}
  f(x,\xi,v)e^{iN P(x,\xi,v)}dxd\xi
  dv&=\int_{\mathbb{R}^3}\int_{T^*\mathbb{R}^3}\chi(x,\xi)
      f(x,\xi,v)e^{iN P(x,\xi,v)}dxd\xi dv \nonumber \\
  &\qquad +\int_{\mathbb{R}^3}\int_{T^*\mathbb{R}^3}(1-\chi(x,\xi)) f(x,\xi,v)e^{iN P(x,\xi,v)}dxd\xi dv. \label{eq:newish'3}
\end{align}
The second integral is $O(N^{-\infty})$ by the same reasoning proceeding \eqref{eq:newish3}. For the first integral, we do a
change of variables. We define the following vectors:
$$x(\beta) \coloneqq \begin{psmallmatrix}\sin \beta-1 \\ 0\\ 0 \end{psmallmatrix},\ \ \xi(\beta) \coloneqq \tfrac{1}{1-\sin \beta}\begin{psmallmatrix}\cos \beta\\ 0\\
                 0 \end{psmallmatrix},\ \ m_{\beta} \coloneqq
                 c_{\beta}\begin{psmallmatrix}\frac{1}{\sin
                            \beta-1} \\ 0\\
                            0 \end{psmallmatrix},\ \ m_{\beta}'
                     \coloneqq
                            c_{\beta}\begin{psmallmatrix}\cos
                                       \beta \\ 0\\
                                       0 \end{psmallmatrix} $$
                                       where $c_{\beta}
                                       \coloneqq (\cos^2
                                       \beta+\frac{1}{(1-\sin
                                        \beta)^2})^{-1/2}$ is a
                                       normalization factor. Now
                                       we do the change of
                                       variables
                                       $(x,\xi,v) \to
       (\beta,t_1,t_2,s_1,s_2,s_3,s_1',s_2',s_3')$
                                       where
                                       $$x=x(\beta)+t_1e_2+t_2e_3+s_1m_{\beta},
                                       \quad
                                       \xi=\xi(\beta)+s_1m_{\beta}'+s_2e_2+s_3e_3,
                                       \quad
                                       v=s_1'm_{\beta}'+s_2'e_2+s_3'e_3.$$
On $\mathcal{C}$, one can easily compute that $dxd\xi dv=|\cos \beta| dt ds ds'd\beta$. We proceed the same as before: we apply the method of stationary phase in the variables $(t,s,s')$ at the only critical point $(0,0,0)$. The Hessian is very similar to \eqref{eq:HESS} (in fact, this case is easier as the block matrices are diagonal), and one can compute that \begin{equation*}
\sqrt{|\det \operatorname{Hess}(\widetilde{P})_{\beta}|}=2(1-\sin \beta)^{3}|\cos \beta|.
\end{equation*}
We then have
           \begin{align*}
&\frac{(N+1)^4}{16\pi^5}\int_{\mathbb{R}^3}\int_{T^*\mathbb{R}^3}
  f(x,\xi,v)e^{iN P(x,\xi,v)}dxd\xi dv\\
  &=\frac{(N+1)^4}{16\pi^5}\Big(\frac{2\pi}{N}\Big)^4\int_{-3\pi/2}^{\pi/2}\frac{16a(\tfrac{1}{p_0^2}x(\beta),p_0
    \xi(\beta))}{(|\xi(\beta)|^2+1)^4}\frac{|\cos \beta|}{2(1-\sin \beta)^3|\cos \beta|}d\beta+O(\tfrac{1}{N}) \\
  &=\frac{1}{2\pi} \int_{-3\pi/2}^{\pi/2}a(\tfrac{1}{p_0^2}x(\beta),p_0
    \xi(\beta))(1-\sin \beta)d\beta+O(\tfrac{1}{N})\\
  &=\frac{p_0^3}{2\pi} \int_{t_{\gamma}-2\pi/p_0^3}^{t_{\gamma}}a\big(\gamma(t)\big)dt+O(\tfrac{1}{N}),
           \end{align*}
where the last line we change variables $\beta \to t$ where $t$
is as in Theorem \ref{theo:moser} and $t_{\gamma}$ is the collision time (defined in Remark \ref{rem:collision}).         
\section{Proof of Theorem \ref{theo:2}}
We start by viewing the resulting integral as an integral on
$\overline{\Sigma_E}$ (see \eqref{eq:reghamil}). Indeed, since $a \in C_c(\Sigma_E)$,
\begin{equation}\label{eq:equivalent}
\int_{\Sigma_E}a d\mu=\int_{\Sigma_E}\overline{a}\big( i_{\Sigma_E}(x,\xi)\big) d \mu(x,\xi)=\int_{\overline{\Sigma_E}}\overline{a} d \overline{\mu},
\end{equation}
where $\overline{a} \in C(\overline{\Sigma_E})$ is defined in \eqref{eq:barr}, and $\overline{\mu} \coloneqq (i_{\Sigma_E})_*\mu$. Now we view this integral as an integral over oriented regularized Kepler orbits. This space is
$\mathcal{H}(\overline{\Sigma_E}) \coloneqq \overline{\Sigma_E}/\sim$ where we quotient out by points on the same regularized Kepler orbit (see Remark \ref{rem:bars}). By \eqref{eq:reghamil},
$$\mathcal{H}(\overline{\Sigma_E}) \cong T_1^*\mathbb{S}^3/ \mathbb{S}^1=\operatorname{SO}(4)/(\operatorname{SO}(2) \times \operatorname{SO}(2))=\widetilde{\textbf{Gr}}(2,4), $$
where $\widetilde{\textbf{Gr}}(2,4)$ is the oriented Grassmanian
manifold (i.e. the double cover of $\textbf{Gr}(2,4)$). That is,
the space of regularized Kepler orbits on $\overline{\Sigma_E}$ is the same as the space of
geodesics on $\mathbb{S}^3$. In particular, the space of regularized Kepler orbits is a
compact manifold. If we denote $\pi:\overline{\Sigma_E} \to \mathcal{H}(\overline{\Sigma_E})$
the projection, then the disintegration theorem says
\begin{equation}\label{eq:disint}
\int_{\overline{\Sigma_E}}\overline{a}d \overline{\mu}= \int_{\mathcal{H}(\overline{\Sigma_E})}\bigg(\int_{\pi^{-1}(\overline{\gamma})}\overline{a}d\nu_{\overline{\gamma}}\bigg) d(\pi_{*}\overline{\mu}) (\overline{\gamma}),
\end{equation}
where $\nu_{\overline{\gamma}}$ are
probability measures on $\overline{\Sigma_E}$ such that
$\operatorname{supp}\nu_{\overline{\gamma}} \subseteq \pi^{-1}(\overline{\gamma})$ for
$\pi_*\overline{\mu}$-almost all $\overline{\gamma} \in \mathcal{H}(\overline{\Sigma_E})$ (see \cite[III-70]{DM78} for the disintegration theorem). Note that
\eqref{eq:disint} is true for merely $\overline{\mu}$-integrable $\overline{a}$ (in particular, indicator functions supported on orbits), so
since $\overline{\mu}$ is invariant under the regularized Hamiltonian flow (by assumption), we see that $\nu_{\overline{\gamma}}$ is invariant under the regularized Kepler
flow for $\pi_*\overline{\mu}$-almost all $\overline{\gamma}\in \mathcal{H}(\overline{\Sigma_E})$. Then for $\pi_*\overline{\mu}$-almost all $\overline{\gamma}$,
$$\int_{\pi^{-1}(\overline{\gamma})}\overline{a} d\nu_{\overline{\gamma}} =\mathcal{R}[\overline{a}](\overline{\gamma})\coloneqq \frac{p_0^3}{2\pi}\int_{0}^{2\pi/p_0^3}\overline{a}(\overline{\gamma}(t))dt,$$
where
$\mathcal{R}[\overline{a}]\in C(\mathcal{H}(\overline{\Sigma_E}))$ is
the Radon transform. With
\eqref{eq:disint}, this implies
\begin{equation}\label{eq:disint2}
\int_{\overline{\Sigma_E}}\overline{a} d \overline{\mu}= \int_{\mathcal{H}(\overline{\Sigma_E})} \mathcal{R}[\overline{a}](\overline{\gamma}) d (\pi_*\overline{\mu})(\overline{\gamma}).
\end{equation}
On the other hand, for any $\overline{\gamma_0} \in \mathcal{H}(\overline{\Sigma_E})$, Theorem \ref{theo:1} and \eqref{eq:seelater} give
\begin{equation}\label{eq:steppingstone}
\langle \operatorname{Op}_{\hbar}(a)\Psi_{\hbar,N}^{\gamma_{0}},\Psi_{\hbar,N}^{\gamma_0} \rangle \to \mathcal{R}[\overline{a}](\overline{\gamma_0})=\delta_{\overline{\gamma_0}}\big[\mathcal{R}[\overline{a}]\big] \coloneqq \int_{\mathcal{H}(\overline{\Sigma_E})}\mathcal{R}[\overline{a}](\overline{\gamma}) d\delta_{\overline{\gamma_0}}(\overline{\gamma}).
\end{equation}
Now we would like to show the analogous statement to \eqref{eq:steppingstone} for convex
combinations of delta masses. Let $c_1,\ldots,c_n \in [0,1]$ be
such that $c_j>0$ and $\sum c_j=1$. Let
$\overline{\gamma_1},\ldots,\overline{\gamma_n} \in \mathcal{H}(\overline{\Sigma_E})$ be distinct
regularized Kepler orbits. Then consider $\Psi_{\hbar,N} \coloneqq
\sqrt{c_1}\Psi_{\hbar,N}^{\gamma_1}+\cdots+\sqrt{c_n}
\Psi_{\hbar,N}^{\gamma_n}$. We claim
\begin{equation}\label{eq:convex}
\langle \operatorname{Op}_{\hbar}(a)\Psi_{\hbar,N},\Psi_{\hbar,N} \rangle\to \sum_jc_j\delta_{\overline{\gamma_j}}(\mathcal{R}[\overline{a}]).
\end{equation}
Indeed, this follows immediately from \eqref{eq:steppingstone}
and the fact that $\langle
\operatorname{Op}_{\hbar}(a)\Psi_{\hbar,N}^{\gamma_j},\Psi_{\hbar,N}^{\gamma_k}
\rangle\to 0$ for $j \neq k$, which we prove after this argument in Lemma \ref{lem:lem} (the Coulomb analog of Lemma 2.1 in \cite{TV-B97}). It is well-known (by the Krein-Milman theorem) that convex combinations of delta measures are weak-* dense in the compact, convex set of probability measures on $\mathcal{H} (\overline{\Sigma_E})$ (equipped with the weak-* topology). We can find eigenfunctions whose semiclassical limit coincides with any given convex combination of delta measures applied to $\mathcal{R}[\overline{a}]$ by \eqref{eq:convex}, so we are done by \eqref{eq:disint2}.
\begin{lem}\label{lem:lem} For $a \in C_c^{\infty}(T^*\mathbb{R}^3)$ and $\overline{\gamma}\neq  \overline{\gamma'} \in \mathcal{H}(\overline{\Sigma_E})$, we
have $\langle
\operatorname{Op}_{\hbar}(a)\Psi_{\hbar,N}^{\gamma},\Psi_{\hbar,N}^{\gamma'}
\rangle \to 0$ as $\hbar \to 0,N \to \infty$ while $E_N(\hbar) \to E$.
\end{lem}
The argument is the same given in \cite[Proposition 3.1]{G91}. We include a proof for completeness.
\begin{proof}
Let $\chi ,\chi'\in C_c^{\infty}(T^*\mathbb{R}^3,[0,1])$ be such that
\begin{align*}
&\chi+\chi'=1 \text{ in a neighborhood of }\operatorname{supp}
  a,\\
 &\operatorname{supp} \chi \cap \gamma' = \varnothing,\quad  \operatorname{supp} \chi'\cap \gamma  = \varnothing.
\end{align*}
We have
$\operatorname{Op}_{\hbar}(a\chi)+\operatorname{Op}_{\hbar}(a\chi')=\operatorname{Op}_{\hbar}(a)$. Using basic microlocal analysis tools (see \cite[(4.1.12), Theorem 4.24]{Z12}) and the Cauchy-Schwarz inequality,
\begin{align*}
|\langle
\operatorname{Op}_{\hbar}(a)\Psi_{\hbar,N}^{\gamma},\Psi_{\hbar,N}^{\gamma'}
\rangle|&= |\langle
\operatorname{Op}_{\hbar}(a\chi')\Psi_{\hbar,N}^{\gamma},\Psi_{\hbar,N}^{\gamma'}
\rangle+\langle
\Psi_{\hbar,N}^{\gamma},\operatorname{Op}_{\hbar}(a\chi)^*\Psi_{\hbar,N}^{\gamma'}
          \rangle|\\
  &\leq \lVert
    \operatorname{Op}_{\hbar}(a\chi')\Psi_{\hbar,N}^{\gamma}
    \rVert_{L^2}^2+\lVert
    \operatorname{Op}_{\hbar}(a^*\chi)\Psi_{\hbar,N}^{\gamma'}
    \rVert_{L^2}^2,\\
        &= \langle
\operatorname{Op}_{\hbar}(|a\chi'|^2)\Psi_{\hbar,N}^{\gamma},\Psi_{\hbar,N}^{\gamma}
\rangle+\langle
\operatorname{Op}_{\hbar}(|a\chi|^2)\Psi_{\hbar,N}^{\gamma'},\Psi_{\hbar,N}^{\gamma'}
\rangle+O(\hbar).
\end{align*}
Taking the limit on both sides and invoking Theorem \ref{theo:1} yields the result.

\end{proof}

\bibliographystyle{alpha-reverse}
\bibliography{refs}

\end{document}